\documentclass[10.9pt,a4paper]{amsart}
\usepackage{a4wide}
\usepackage[utf8]{inputenc}
\usepackage[english]{babel}

\usepackage{amsfonts,amssymb,amsmath}

\usepackage[nobysame,alphabetic, initials]{amsrefs}

\usepackage{comment}

\usepackage{graphicx}

\usepackage[dvipsnames]{xcolor}
\usepackage[colorlinks,
    linkcolor={blue!50!black},
    citecolor={blue!50!black},
    urlcolor={black!80!black}]{hyperref}
 \usepackage{cleveref}

 \usepackage{enumerate}
\usepackage{tikz}
\usepackage[all]{xy}
\usetikzlibrary{calc, matrix, arrows, cd, decorations.markings, knots}
\usetikzlibrary{decorations.pathmorphing}
\usetikzlibrary{decorations.pathreplacing}
\usepackage{thmtools}
\usepackage{thm-restate}
\usepackage{amsthm}

\newcommand{\customgenericname}{} 

\newtheorem{innercustomgeneric}{\customgenericname}

\newcommand{\newcustomtheorem}[2]{%
  \newenvironment{#1}[1]{%
    \renewcommand{\customgenericname}{#2}%
    \renewcommand{\theinnercustomgeneric}{##1}%
    \begin{innercustomgeneric}%
  }{%
    \end{innercustomgeneric}%
  }%
}

\newcustomtheorem{customthm}{Theorem}
\newcustomtheorem{customlemma}{Lemma}

\declaretheorem[name=Theorem,numberwithin=section]{thm}
\newtheorem{theorem}{Theorem}[section]

\newtheorem{corollary}[theorem]{Corollary}
\newtheorem{lemma}[theorem]{Lemma}
\newtheorem{proposition}[theorem]{Proposition}

\theoremstyle{definition}
\newtheorem{definition}[theorem]{Definition}

\newtheorem{notation}{Notation}

\newtheorem{remark}[theorem]{Remark}

\newtheorem{claim}{Claim}
\newtheorem*{claim*}{Claim}

\newcommand{\s}{\mathfrak{s}}

\newcommand{\Z}{\mathbb{Z}}

\newcommand{\R}{\mathbb{R}}

\newcommand{\Hom}{\operatorname{Hom}}

\newcommand{\im}{\operatorname{Im}}

\newcommand{\Spin}{\operatorname{Spin}}

\newcommand{\bsm}{\left(\begin{smallmatrix}}
\newcommand{\esm}{\end{smallmatrix}\right)}

\definecolor{bettergreen}{rgb}{0.0, 0.5, 0.0}

\RequirePackage{tikz-cd}
\RequirePackage{amssymb}
\usetikzlibrary{calc}
\usetikzlibrary{decorations.pathmorphing}

\tikzset{curve/.style={settings={#1},to path={(\tikztostart)
    .. controls ($(\tikztostart)!\pv{pos}!(\tikztotarget)!\pv{height}!270:(\tikztotarget)$)
    and ($(\tikztostart)!1-\pv{pos}!(\tikztotarget)!\pv{height}!270:(\tikztotarget)$)
    .. (\tikztotarget)\tikztonodes}},
    settings/.code={\tikzset{quiver/.cd,#1}
        \def\pv##1{\pgfkeysvalueof{/tikz/quiver/##1}}},
    quiver/.cd,pos/.initial=0.35,height/.initial=0}

\tikzset{tail reversed/.code={\pgfsetarrowsstart{tikzcd to}}}
\tikzset{2tail/.code={\pgfsetarrowsstart{Implies[reversed]}}}
\tikzset{2tail reversed/.code={\pgfsetarrowsstart{Implies}}}
\tikzset{no body/.style={/tikz/dash pattern=on 0 off 1mm}}

\begin{document}
\title{Knotted surfaces with simply-connected complements}

\author[A.~Pyronneau]{Audrick Pyronneau}
\address{The University of Texas at Austin, Austin TX 78712}
\email{audrickpyro@utexas.edu}

\def\subjclassname{\textup{2020} Mathematics Subject Classification}
\expandafter\let\csname subjclassname@1991\endcsname=\subjclassname

\begin{abstract}
We study locally flat, compact, oriented, genus $g$, surfaces with boundary a fixed knot $K,$ properly embedded in simply-connected 4-manifolds with boundary the 3-sphere, and whose complements have trivial fundamental group. We show that if two such surfaces are homologous then they are topologically ambiently isotopic rel. boundary. 
 \end{abstract}

\maketitle

\section{Introduction}
\label{sec:Introduction}

Given a closed simply-connected 4-manifold $X,$ Boyer \cite{Boyer1993} proved that closed genus $g$ orientable surfaces $F_1,F_2\subset X$ with $\pi_1(X\setminus F_i)=1$ are isotopic  if they are homologous. The goal of this article is to generalize this result to surfaces with boundary a knot $K\subset S^3.$ Building on these ideas and using $X$ to denote a compact, simply-connected 4-manifold with boundary $S^3,$ we prove the following uniqueness result:

\begin{thm}\label{mainthm}
   Let $K$ be a knot in $S^3.$ Any two homologous orientable genus $g$ surfaces $F_1,F_2\subset X,$ both with boundary $K$ and $\pi_1(X\setminus F_i)=1$ for $i=1,2,$ are topologically ambiently isotopic rel. boundary. 
\end{thm}

Recently, Conway, Orson, and Pencovitch in \cite{conway2025simplyslicingknots}  proved an existence result giving conditions for a knot to bound a locally flat disk in a simply-connected 4-manifold with boundary $S^3,$ whose complement has finite cyclic fundamental group. Combining \cite[Corollary 1.3]{conway2025simplyslicingknots} and \cite[Corollary 1.7]{KPRT2024} with Theorem~\ref{mainthm} gives the following:
\begin{corollary}
Let $K \subset S^3$ be a knot.
Mapping a surface $F\subset X$ to its relative homology class gives the following
bijections:

\footnotesize
For $g>0$,
\[
\left\{
F \subset X \ \middle|\
\begin{array}{c}
g(F)=g,\ \partial F=K,\\
\pi_1(X\setminus F)=1
\end{array}
\right\}
\Big/{\sim}
\longleftrightarrow
\left\{
x\in H_2(X,\partial X)\ \middle|\ x \text{ is primitive}
\right\}.
\]

For $g=0$,
\[
\left\{
F \subset X \ \middle|\
\begin{array}{c}
F \text{ is a disk},\ \partial F=K,\\
\pi_1(X\setminus F)=1
\end{array}
\right\}
\Big/{\sim}
\longleftrightarrow
\left\{
x\in H_2(X,\partial X)\ \middle|\
\begin{array}{c}
x \text{ is primitive, and if $x$ is characteristic},\\[2pt]
\displaystyle
\operatorname{Arf}(K)+KS(X)
+\frac{\sigma(X)-x\cdot x}{8}
\equiv 0 \pmod 2
\end{array}
\right\}.
\]
Here $\sim$ denotes topological ambient isotopy rel.\ boundary.

\end{corollary}

Lastly, a surface $F\subset X^4$ is said to be \emph{topologically flexible} if every element of the mapping class group can be extended as an ambient homeomorphism of $X^4.$ Additionally, a characteristic surface $F\subset X^4$ is said to be \emph{topologically spin-flexible} if every element of its mapping class group preserving the \emph{Rokhlin quadratic form} (see Definition~\ref{rokh}), can be extended as an ambient homeomorphism of $X^4.$ As a consequence of the proof of Theorem \ref{mainthm}, a genus $g$ surface $F\subset X$ with boundary $K\subset S^3,$ and $\pi_1(X\setminus F)=1$ is topologically flexible or spin-flexible respectively. 
\begin{customthm}{Theorem 4.3}
 Let $K\subset S^3$ be a knot and $F\subset X$ a genus $g$ surface with boundary $K$ and $\pi_1(X\setminus F)=1.$ Fix an orientation-preserving homeomorphism $h:F\to F$ restricting to the identity on $\partial F.$ 
    \begin{itemize}
        \item If $F$ is characteristic, and $h$ preserves the Rokhlin form $q_{FK}\circ h_{\ast}=q_{FK}$ then $h$ extends to a homeomorphism of pairs $H:(X,F)\xrightarrow[]{\cong}(X,F).$\item When $F$ is ordinary, $h$ extends as a homeomorphism of pairs $H:(X,F)\xrightarrow[]{\cong}(X,F).$ 
        \end{itemize}
    \end{customthm}
Now we give a condensed outline of the proof of Theorem \ref{mainthm}.
Let $\nu F$ be an open tubular neighborhood, and $X_F:=X\setminus \nu F$ be the surface exterior.
\begin{enumerate}
    \item Start with a homeomorphism $h:F_1\to F_2$ that restricts to the identity on $\partial F_1=K.$ Since any homeomorphism of a surface lifts to a bundle isomorphism of the trivial $D^2-$bundle, extend $h$ to a bundle isomorphism $H:(\overline{\nu}F_1,F_1)\to (\overline{\nu}F_2,F_2).$ Now using $H$ we define $f:\partial X_{F_1}\to \partial X_{F_2},$ a homeomorphism over the surface exterior boundaries.
    \item Next we build an isometry on 2nd homology (see Proposition \ref{main1} and Lemma \ref{main2}) $\Lambda:(H_2(X_{F_1}),Q_{X_{F_1}})\xrightarrow[]{\cong}(H_2(X_{F_2}),Q_{X_{F_2}})$ which is compatible with $f$ (see Definition~\ref{pair}), and apply \cite[Theorem 0.7 and Proposition 0.8]{boyer1986simply} to extend $f$ to $F:X_{F_1}\to X_{F_2}$ over the surface exteriors. 
    \item To finish, we define a $\hat{F}=F\cup H:(X,F_1)\to (X,F_2)$ which restricts to $id_{S^3}.$  \cite[Corollary C]{orson2022mapping} states that if $\hat{F}|_{S^3}=id_{S^3}$ and $\hat{F}_{\ast}|_{H_2(X)}=id_{H_2(X)}$ then $\hat{F}$ is isotopic to the identity rel. boundary. We verify this result in Proposition \ref{main3}. 
\end{enumerate}

\subsection*{Previous uniqueness results for oriented, simple surfaces and flexibility}
Lastly, we give a brief summary of surface classification results in topological $4-$manifolds. The \emph{unknotting conjecture} for locally flat surfaces in $S^4$ was first proven in the $g=0$ case by Freedman in \cite[Theorem 7]{Freedman1984DiskTheorem} (see also \cite[Theorem 11.7A]{FreedmanQuinn1990}). Furthermore, Conway and Powell prove the unknotting conjecture given $g\geq3$ \cite[Theorem 1.1]{Conway_2023}. Moreover, Lee and Wilczyński \cite{LeeWilczynski1990LocallyFlatSpheres,LeeWilczynski1993RepresentingHomology,LeeWilczynski1997} proved existence and uniqueness results for locally flat surfaces in closed simply-connected 4-manifolds representing a fixed homology class. The aforementioned \cite{Boyer1993} also provides existence for closed surfaces of all genera when the surface exterior is simply-connected. Finally, \cite{ConwayPowell2021HomotopyRibbonDiscs,Conway_2023,conway20244manifoldsboundaryfundamentalgroup, conway2025simplyslicingknots} give existence and uniqueness results for $\Z-$surfaces with boundary.

\cite{Hirose2002Diffeomorphisms} introduces flexibility of closed surfaces embedded in
$4$-manifolds. For unknotted orientable surfaces in $S^4$, he
showed that extendibility of surface diffeomorphisms  is detected by preservation of the Rokhlin
quadratic form. More recently, Lehman \cite{lehman2026flexiblesurfacesmathbbcp2s2times} constructs flexible surfaces and spin-flexible surfaces in $\mathbb{CP}^2$ and $S^2 \times S^2$ within any prescribed ordinary or characteristic homology class.
\subsection*{Organization}
In Section \ref{sec:AlgTop}, we collect some facts about the algebraic topology of surface exteriors. In Section \ref{Spinor}, we give conditions when given $F_1,F_2\subset X$ it is possible to form the union $M := X_{F_1}\cup -X_{F_2}$ along a homeomorphism of the boundaries so that $M$ is a spin 4-manifold. Finally, in Section \ref{end}, we give the proof of the main theorem. 
\subsection*{Conventions}
We work in the topological category. Embeddings are assumed to be locally flat and manifolds are assumed to be compact, connected, and oriented, with connected boundary.
\subsection*{Acknowledgments}
The author is grateful to his advisor, Anthony Conway, for helpful conversations, feedback, and support. The author was supported by the NSF via the Graduate Research Fellowship Program.
\section{The algebraic topology of surface exteriors} \label{sec:AlgTop}
Here we recall a few facts on the algebraic topology of simply-connected surface exteriors. Fix $X$ to be a simply-connected 4-manifold with boundary homeomorphic to $S^3.$
Let $X_F := X \setminus \nu F$ denote the surface exterior.

\subsection{The homology of $\partial X_F$} 
Let $F\subseteq X$ be an embedded orientable surface with boundary a knot $K\subset S^3.$ Note the boundary of the exterior $\partial X_F$ decomposes as $E_K\cup -\mathring{Y},$ for $E_K$ the knot exterior of $K\subseteq S^3$ and $\mathring{Y}:=(\partial\overline{\nu}F\setminus \overline{\nu}(\partial F))$ an oriented 3-manifold with the structure of a trivial $S^1-$ bundle over $F.$ 
Now we compute the homology groups of $\mathring{Y}.$ Abstractly $\mathring{Y}$ is homeomorphic to $S(\overline{\nu}F),$ the unit sphere bundle of $\overline{\nu}F.$

\begin{lemma}\label{HomologyY}
The nontrivial homology groups of $\mathring{Y}$ are $$H_0(\mathring{Y})\cong\Z,\;\;\;H_1(\mathring{Y})\cong \Z^{2g}\oplus \Z,\text{ and }H_2(\mathring{Y})\cong \Z^{2g}. $$

\end{lemma}

\begin{proof}
Recall that an oriented rank-2 vector bundle over a paracompact base space $F$ is trivial if and only if its Euler class vanishes. Since $F$ has nonempty boundary, it is homotopy equivalent to a 1-complex; in particular $H^2(F)=0,$ so $e(\overline{\nu} F)=0.$ We deduce that $\overline{\nu}{F}$ is the trivial rank-2 disk bundle over $F.$ It follows from the definition of $\mathring {Y}$ that $\mathring{Y}\cong F\times S^1.$    
The result now is a consequence of the Künneth formula. 
\end{proof}

Recall that the Euler number of a surface $F\subset X$ relative to a framing $s'$ of its boundary or \textit{relative Euler number} can be calculated by pairing the fundamental class of the surface with the \textit{relative Euler class} of its normal bundle $\nu_{F\subset X}$:  $$e(\nu_{F\subset X},s'):=[F]\frown \epsilon(\nu_{F\subset X},s')\in H_0(F)=\Z.$$ Here the \textit{relative Euler class} is the obstruction to finding a section $s$ of the sphere bundle $S(\nu F)$ extending $s'.$ For a more detailed discussion, see \cite[Section 2]{conway2025unknotting}.
\begin{remark}
In practice, the (relative) Euler number of $F$ equals the algebraic intersection number of a generic section with the zero section in the total space of $\overline{\nu}F$.
\end{remark}

 We now establish the relationship between the relative Euler number of a surface with boundary to the Euler number of the closed surface obtained by capping it off. We use $e(F,\epsilon_S)$ to denote the relative Euler number of $F\subset X$ with respect to the Seifert framing of its boundary.

\begin{lemma}\label{e=lk}The relative Euler number $e:=e(F,\epsilon_S)$ equals the Euler number $e(\hat{F})$ of the closed surface $(\hat{X},\hat{F})$ obtained by gluing $(X,F) $ and $ (B^4,G)$, where $G$ is a surface with boundary $K$. Moreover, the decomposition of $\partial X_F$ as $E_K\cup \mathring{Y},$ is given by identifying the $S^1-$fiber of $\mathring{Y}$ with the meridian of $E_K,$ and a section $s:F\to \mathring{Y},$ identifying $s(\partial F)$ with the $e-$framed longitude of $K.$ \end{lemma}
\begin{proof}
    Let $G$ be a genus $g$ surface in $B^4$ with boundary $K\subset S^3.$ Pick push-offs $F^+\subset X$ and $G^+\subset B^4$ whose boundaries are the 0-framed longitude, $\lambda_K.$ Then $$F^+ \cup_{\lambda_K} G^+\subset  \hat{X}$$ is a push-off of $$\hat{F}=F\cup_{K} G$$ in $\hat{X}$. Using the definition of relative Euler number we have the following equalities: $$e(\hat{F})=\hat{F}\cdot\hat{F}^+=F\cdot F^++G\cdot G^+=e(F,\epsilon_{S})+lk (K,\lambda_K)=e(F,\epsilon_{S})+0.$$ 

    Lastly, recall $\mathring{Y}\cong F\times S^1$. Geometrically one sees that the $S^1$-fiber is identified with the meridian of $E_K$, and $s(\partial F)$ is identified with an $n$-framed longitude of $E_K$ for some $n$. Now choose new push offs $F^+=s(F)\subset X$ and $G^+\subset B^4$ of $G$ whose boundaries are the $n-$framed longitude. The framing coefficient $n$ is given by
\[
n=lk(\partial G,\partial G^+)=G\cdot \hat{G}^+.
\]
Therefore, since $\hat{F}^+=F^+\cup_{\lambda_K}G^+$ is a push-off of $\hat{F}$,
$$ n=\overbrace{F\cdot F^+}^{=0} +G\cdot G^+ =e(\hat{F}) =e(F,\epsilon_S),$$
where the final equality was established above. This completes the proof.
\end{proof}
For simplicity let $e(F)=e(F,\epsilon_S).$ Now we state one of the main results of this section.
\begin{proposition}\label{BXF}
 The nontrivial homology groups of $\partial X_F$ for $F$ with relative Euler number $e(F)=e,$ are $$H_0(\partial X_F)\cong\Z,\;H_1(\partial X_F)\cong \Z^{2g}\oplus \Z_e,\;H_2(\partial X_F)\cong\left\{ \begin{array}{rcl}
         \Z^{2g} & \mbox{for} & e\neq 0 \\ 
         \Z^{2g}\oplus \Z & \mbox{for} & e=0   
                \end{array}\right.\text{ and  }H_3(\partial X_F)\cong \Z\ .$$
\end{proposition}
\begin{proof}
 $\partial X_F$ is a closed connected oriented 3-manifold, so $H_0(\partial X_F)$ and $H_3(\partial X_F)$ are both infinite cyclic. Now we calculate $H_1(\partial X_F)$ via the Mayer-Vietoris sequence with respect to decomposition $E_K\cup \mathring{Y}:$
$$\cdots \to H_1(E_K\cap \mathring Y)\to H_1(E_K)\oplus H_1(\mathring{Y})\to H_1(\partial X_F)\to 0.$$

The maps induced by inclusion of $\mathring{Y}\cap E_K$ into $E_K$ and $\mathring{Y}$ respectively, determine the first and second homology groups of $\partial X_F.$ Using Lemma \ref{e=lk}, we identify the $S^1-$fiber of $\mathring{Y}$ with the meridian of $E_K,$ and for a section $s(\partial F)$ with an $e-$framed longitude. Thus, 
\begin{enumerate}
    \item For $H_1(E_K\cap\mathring{Y})\to H_1(E_K),$ the meridian of $E_K\cap\mathring{Y},$ $m$ maps homologically to the generator of $H_1(E_K)$ while $l$ maps to $e[\mu_K]$.
    \item For $H_1(E_K\cap\mathring{Y})\to H_1(\mathring{Y}),$  $m$ maps to the ``$S^1-$fibre'' of $\mathring{Y}$ denoted by $[\mu]\in H_1(\mathring{Y})$ while the longitude of $E_K\cap\mathring{Y},$ $l$ maps to zero.
\end{enumerate}
We apply the first isomorphism theorem to verify the result for $H_1(\partial X_F).$ We apply Poincaré duality followed by the universal coefficient theorem to conclude the result.

\end{proof}
\subsection{The intersection form of a surface exterior} 

Fix a simply-connected 4-manifold $X$ with boundary homeomorphic to $S^3$ and an embedded surface $F\subseteq X$ with connected boundary such that $\pi_1(X_F)=1.$ This section is concerned with describing the intersection form of a simply connected surface exterior, $Q_{X_F}.$ 

\begin{lemma}\label{Im=E}
    Given $F\subset X $ such that $\pi_1(X_F)=1,$ the image of the inclusion induced map $i_{\ast}:H_2(X_F)\to H_2(X)$ is $\{\beta\in H_2(X):Q_X(\beta, [F])=0\}.$ 
\end{lemma}
\begin{proof}
    Consider the Mayer–Vietoris sequence for the decomposition $X = X_F \cup \overline{\nu}F$:
    \[
        \cdots \to H_2(X_F) \oplus H_2(\nu F) \xrightarrow{i_\ast} H_2(X) \xrightarrow{\partial} H_1(X_F \cap \nu F) \to \overbrace{H_1(X_F)}^{=0}\oplus H_1(\nu F)\to 0 .
    \]
    First, observe that $\nu F$ deformation retracts to $F$, so $H_2(\nu F) \cong H_2(F) = 0.$ Therefore, by exactness $\operatorname{Im}(i_\ast) = \ker(\partial).$ It therefore remains to analyze $\partial.$ We begin by determining the image of $\partial.$ Observe that $X_F \cap \nu F \cong \mathring{Y}$, so by Lemma~\ref{HomologyY}  $H_1(X_F \cap \nu F)\cong H_1(F)\oplus H_1(S^1).$
     Under inclusion, the $H_1(F)$ summand maps isomorphically onto the $H_1(\nu F)$ term. Hence by exactness the image of $\partial$ maps isomorphically onto the $H_1(S^1)$ factor generated by the meridian of $F.$  Hence, we conclude that the above sequence simplifies to \[H_2(X_F)\xrightarrow[]{i_{\ast}}H_2(X)\xrightarrow[]{\partial}\Z\langle\mu\rangle\to 0.\] Finally, a second homology class $\beta\in H_2(X)$ can be represented by a closed surface $\Sigma\subset X.$ By the geometric description of the connecting homomorphism in the Mayer-Vietoris sequence, $\partial([\Sigma])=Q_X([\Sigma],[F]).$ We deduce that $\{\beta \in H_2(X):Q_X(\beta, [F])=0\}= \ker(\partial)=\text{Im}(i_{\ast})$ concluding the result. 
\end{proof}

For simplicity, given $\alpha\in H_2(X,\partial X)$ write $E(\alpha):=\{x\in H_2(X):Q_X(x, \alpha)=0\}.$ 
The next lemma is a preliminary result needed to calculate $Q_{X_{F}}.$
\begin{lemma}\label{surj}
 Suppose the relative Euler number $e(F)$ is non-zero and $\pi_1(X_F)=1.$ For the inclusion induced maps $$H_2(\partial X_F)\xrightarrow[]{i_{\partial X_F}} H_2(X_F)\xrightarrow[]{i_{\ast}}H_2(X)$$  the following inclusion holds: $$\operatorname{Im}(i_{\partial X_F})\subseteq \ker(i_{\ast}).$$
\end{lemma}
\begin{proof}
    We consider the following commutative diagram: 
    \begin{equation}\label{diagram1}
\begin{tikzcd}[ampersand replacement=\&,cramped]
	0 \& {H_3(X,X_F)} \&\&\&\& {H_2(X_F)} \&\& {H_2(X).} \\
	\\
	\& {H_3(\nu F,\mathring{Y})} \&\& {H_2(\mathring{Y})} \&\& {H_2(\partial X_F)}
	\arrow[from=1-1, to=1-2]
	\arrow["{{{\partial _{X_F}}}}", from=1-2, to=1-6]
	\arrow["{{{(exc)^{-1}}}}", from=1-2, to=3-2]
	\arrow["\cong"', from=1-2, to=3-2]
	\arrow["{{{i_\ast{}}}}", from=1-6, to=1-8]
	\arrow["{{{\partial_{\mathring{Y}}}}}", from=3-2, to=3-4]
	\arrow["{{{i_{\mathring{Y}}}}}", from=3-4, to=3-6]
	\arrow["{{{i_{\partial X_F}}}}", from=3-6, to=1-6]
\end{tikzcd}
\end{equation}
Recall that the top row is exact. $i_{\partial X_F}$ fits into the long exact sequence of the pair $(X_F,\partial X_F).$ $i_{\partial X_F}$ is injective since by Poincaré-Lefschetz duality $H_3(X_F,\partial X_F)=H^1(X_F)=0$ since $\pi_1(X_F)=1.$

To show $i_{\mathring{Y}}$ is an isomorphism, we analyze $H_2(\partial X_F)$ using the decomposition $ \partial X_F=\mathring{Y}\cup E_K.$ Consider the Mayer-Vietoris long exact sequence:
$$H_3(\partial X_F)\to H_2(E_K\cap \mathring{Y} )\to H_2(\mathring{Y})\oplus \underbrace{H_2(E_K)}_{=0}\xrightarrow[]{(i_{\mathring{Y}})_{\ast}} H_2(\partial X_F)\to H_1(\mathring{Y}\cap E_K)\to H_1(\mathring{Y})\oplus H_1(E_K).$$ 

Since $E_K$ is a homology circle, $H_2(E_K)=0.$ Additionally, the connecting homomorphism sends the fundamental class of $\partial X_F$ to the fundamental class of $E_K\cap \mathring{Y},$ so the map,
$$H_3(\partial X_F)\to H_2(E_K\cap \mathring{Y} )$$ is an isomorphism. Since $e\neq0$ by the proof of Proposition~\ref{BXF} the map $H_1(\mathring{Y}\cap E_K)\to H_1(\mathring{Y})\oplus H_1(E_K)$ is injective. 
  Since both adjacent maps are zero, exactness forces $i_{\mathring{Y}}$ to be an isomorphism. Since $e\neq0,$ both $(i_{\mathring{Y}})_{\ast}$ and $\partial_{\mathring{Y}}$ are isomorphisms, we conclude that by commutativity of the excision square and surjectivity of  $\partial_{\mathring{Y}}$ $$\text{Im}(i_{\partial X_F})\subseteq \text{Im}(\partial_{ X_F})=\ker(i_{\ast})$$ as desired.  
\end{proof}
\begin{remark}
Note the reverse inclusion $\text{Im}(\partial_{X_F})\subset \text{Im}(i_{\partial X_F})$ is true for every $e$ by the injectivity of $i_{\mathring{Y}}.$  \end{remark}
 Fix $F \subset X$ to be a genus $g$ surface, with boundary a knot, $K\subset S^3,$ representing primitive class $\alpha\in H_2(X,\partial X)$ such that $\pi_1(X_F)=1.$ Now we prove the main result of this section:
\begin{proposition}\label{QX}
The intersection form of $X_F$
splits as follows,
 \[Q_{X_F}\cong Q_X|_{E(\alpha)}\oplus0^{\oplus{2g}}.\]
\end{proposition}
\begin{proof}
   Consider the following short exact sequence from Lemma \ref{Im=E}:
   $$0\to H_3(X,X_F)\xrightarrow[]{\partial_{X_F}}H_2(X_F)\xrightarrow{i_{\ast}}E(\alpha)\to 0.$$
   Since $E(\alpha)$ is free abelian, the short exact sequence splits and $H_2(X_F)\cong E(\alpha)\oplus \mathrm{Im}(\partial_{X_F}).$ We now show $Q_{X_F}|_{\mathrm{Im}(\partial_{X_F})}=0.$ 
   Recall any class coming from $H_2(\partial X_F)$ pairs trivially with every class under $Q_{X_F}.$ By Remark 2.7, $e(F),\; \mathrm{Im}(\partial_{X_F})\subset \mathrm{Im}(i_{\partial X_F}).$ Furthermore, by the long exact sequence of a pair $(\overline{\nu}F,\mathring{Y}),$ we deduce that $H_3(\overline{\nu}F,\mathring{Y})\cong H_2(\mathring{Y})$ is isomorphic to $\Z^{2g}$ by Lemma \ref{HomologyY}. Hence by excision we deduce $H_3(X,X_F)\cong \Z^{2g}$ as well, and we conclude $Q_{X_F}\cong Q_X|_{E(\alpha)}\oplus0^{\oplus{2g}}.$
\end{proof}
We record two facts about $E(\alpha)$ for later use in Section \ref{end}.
\begin{lemma}
    If $e(F)$ is nonzero, $Q_X|_{E(\alpha)}$ is nondegenerate.\label{QXE}
\end{lemma}
\begin{proof}
Let $x\in H_2(X_F)$ be such that  $Q_X(i_{\ast}(x),i_{\ast}(z))=0$ for all $z\in H_2(X_F),$ and recall $E(\alpha)=\operatorname{Im}(i_{\ast}).$ We need to show $i_{\ast}(x)=0$ or equivalently $x\in \ker(i_{\ast}).$ Notice $0=Q_X(i_{\ast}(x),i_{\ast}(z))=Q_{X_F}(x,z)$ for every $z\in H_2(X_F).$ This implies $x\in \mathrm{rad}(Q_{X_F}),$ but by the long exact sequence of the pair $(X_F,\partial X_F),$ $\mathrm{rad}(Q_{X_F})=\mathrm{Im}(i_{\partial X_F}).$ Since $e(F)\neq 0$ by Lemma \ref{surj} $\mathrm{Im}(i_{\partial X_F})\subseteq \ker(i_{\ast})$ therefore, $x\in \ker(i_{\ast})$ concluding the proof.
\end{proof}

\begin{lemma}
    Let $e(F)=0$ and $\tilde{\alpha}:=(\tilde{i}_{\ast})^{-1}(\alpha)$ for $\tilde{i}:H_2(X)\xrightarrow[]{\cong} H_2(X,\partial X).$ There exists a subgroup $W$ of $E(\alpha)$ for which $Q_X|_W$ is nondegenerate, with $E(\alpha) = \langle \tilde\alpha\rangle \oplus W$. \label{QXW}
\end{lemma}
\begin{proof}
Since $e(F)=0$, we have $Q_X(\tilde\alpha,\tilde\alpha)=0$, and hence
$\tilde\alpha\in E(\alpha)$. Since $\tilde\alpha$ is primitive and $E(\alpha)$ is free abelian, choose a splitting $E(\alpha)=\langle \tilde\alpha\rangle\oplus W$ for some subgroup $W.$ 

It remains to show that $Q_X|_W$ is nondegenerate. Suppose $\gamma\in W$ satisfies $Q_X(\gamma,w)=0$ for all $w\in W$. Since $\gamma\in W\subset E(\alpha)$, we also have $Q_X(\gamma,\tilde\alpha)=0.$ Thus $\gamma$ pairs trivially with all of $E(\alpha)=\langle \tilde\alpha\rangle\oplus W.$ By unimodularity of $Q_X$ and primitivity of $\tilde\alpha$, there exists $v\in H_2(X)$ such that $Q_X(v,\tilde\alpha)=1,$ thus $H_2(X)=\langle v\rangle\oplus E(\alpha).$ Hence any $x\in H_2(X)$ can be written as $x=mv+e$, with $m\in\mathbb Z$ and $e\in E(\alpha)$. Let $n=Q_X(\gamma,v)$. Since $\gamma$ pairs trivially with $E(\alpha)$, we have
\[
Q_X(\gamma,x)=mQ_X(\gamma,v)=mn.
\]
On the other hand, since $e\in E(\alpha)=\tilde\alpha^\perp$ and $Q_X(\tilde\alpha,v)=1$,
\[
Q_X(n\tilde\alpha,x)=Q_X(n\tilde\alpha,mv+e)=mn.\] Thus, $Q_X(\gamma,-)=Q_X(n\tilde\alpha,-).$ By unimodularity, $\gamma=n\tilde\alpha$. But $\gamma\in W$ and
$n\tilde\alpha\in \langle \tilde\alpha\rangle$, while $E(\alpha)=\langle\tilde\alpha\rangle\oplus W.$
Therefore $\gamma=0$. Hence the radical of $Q_X|_W$ is zero, so $Q_X|_W$ is nondegenerate.
\end{proof}

\subsubsection{A technical lemma}
This lemma will be invoked throughout the proof of the main theorem. Some preparation is needed before the statement.

Given a 4-manifold with boundary $X,$ fix the notation $Q_X^{\partial}$ to be defined as the pairing $H_2(X,\partial X)\times H_2(X)\to \Z$ given by $(a,b)\mapsto \text{ev}(\text{PD}^{-1}(a))(b).$ For the rest of the article we identify the dual space of the second homology and the 2nd relative homology group via the Poincaré duality and evaluation isomorphisms: $$H_2(X_F,\partial X_F)\cong H^2(X_F)\cong \Hom_{\Z}(H_2(X_F);\Z).$$
An isometry $\Lambda$ between two simply-connected four-manifolds refers to an isomorphism on 2nd homology that preserves the intersection form. $F_1$ and $F_2$ are to be embedded, homologous, genus $g$ surfaces with boundary $K$ such that $\pi_1(X\setminus F_i)=1.$ 

\begin{lemma}\label{kj}
Let for $j=1,2$ $k_j:H_2(X)\to H_2(X_{F_j},\partial X_{F_j})$ be the following composition of maps $$H_2(X)\to H_2(X,\overline{\nu}F_j)\xrightarrow[]{(exc)^{-1}} H_2(X_{F_j},\mathring{Y}_j)\to H_2(X_{F_j},\partial X_{F_j}).$$ Additionally, let $(i_j)_{\ast}:H_2(X_{F_j})\to H_2(X)$ be the induced maps from inclusion. 
    If $\Lambda$ is an isometry such that $(i_2)_{\ast}\circ\Lambda =(i_1)_{\ast}$ then following diagram commutes:
    \[\begin{tikzcd}[ampersand replacement=\&,cramped]
	{H_2(X_{F_1})} \&\& {H_2(X_{F_1},\partial X_{F_1})} \\
	\& {H_2(X)} \\
	{H_2(X_{F_2})} \&\& {H_2(X_{F_2},\partial X_{F_2})}
	\arrow[from=1-1, to=1-3]
	\arrow["{{(i_1)_{\ast}}}", from=1-1, to=2-2]
	\arrow["\Lambda"', from=1-1, to=3-1]
	\arrow["{{k_1}}", from=2-2, to=1-3]
	\arrow["{{k_2}}"', from=2-2, to=3-3]
	\arrow["{{(i_2)_{\ast}}}"', from=3-1, to=2-2]
	\arrow[from=3-1, to=3-3]
	\arrow["{{\Lambda^{\ast}}}"', from=3-3, to=1-3].
\end{tikzcd}\]\label{4.3}
\end{lemma}
\begin{proof}
The following argument is given in \cite[Lemma 4.3]{Boyer1993} for closed surfaces, and we verify the argument for the case with boundary.
Observe, $k_j\circ (i_j)_{\ast}$ (for $j=1,2$) agrees with the horizontal maps in the diagram by construction. By assumption $(i_2)_{\ast}\circ \Lambda=(i_1)_{\ast}.$ Therefore, it is sufficient to verify that $k_1=\Lambda^{\ast}\circ k_2.$ Let $\mu\in H_2(X)$ and $\xi\in H_2(X_{F_1})$ be arbitrary. Then using the definition of $k_j,$ we obtain
  \begin{align*}   
Q_{X_{F_1}}^{\partial}(\Lambda^{\ast}\circ k_2(\mu), \xi) &= Q_{ X_{F_2}}^{\partial}(k_2(\mu), \Lambda(\xi))\\
&= Q_X(\mu ,(i_2)_{\ast}\circ\Lambda(\xi))\\
&= Q_X(\mu, (i_1)_{\ast}(\xi))\\
&=Q_{ X_{F_1}}^{\partial}(k_1(\mu), \xi).
\end{align*}

Hence, $\Lambda^{\ast}\circ k_2=k_1$ since $Q_{X_F}^{\partial}$ is nondegenerate.
\end{proof}
 
\section{A spin union of surface exteriors}\label{Spinor} Fix a simply-connected 4-manifold $X$ with boundary homeomorphic to $S^3.$ The goal of this section is to show that given two homologous genus $g$ surfaces with boundary $K,$ and $\pi_1(X_{F_i})=1\;(\operatorname{for}\;i=1,2),$ if both $X_{F_i}$ are spin, then their exteriors can be glued along a boundary homeomorphism $f$ such that $X_{F_1}\cup_f -X_{F_2}$ is spin. Recall the boundary of the exterior $\partial X_F$ decomposes as $E_K\cup -\mathring{Y},$ for $E_K$ the knot exterior of $K\subseteq S^3$ and $\mathring{Y}=\partial\overline{\nu}F\setminus \overline{\nu}(\partial F)$ an oriented 3-manifold with the structure of a trivial $S^1-$ bundle over $F.$
\begin{definition}   
Let $F_1,F_2\subset X$ be embedded surfaces with boundary $K\subset S^3$ a knot. Given a bundle isomorphism $H:(\overline{\nu}F_1,F_1)\to (\overline{\nu}F_2,F_2),$ we define a homeomorphism on the boundary of the surface exteriors as follows $$H|_{\mathring{Y}}\cup id_{E_K}:\partial X_{F_1}\to \partial X_{F_2}.$$ Since $H|_{F_1}$ restricts to $id_K$ on $\partial F_1=K$, the induced bundle map can be chosen to be  the identity over the normal disk bundle of $K$ in $S^3$. Hence $H$ restricts to the identity on the $\partial \overline{\nu} K,$ and $H|_{\mathring{Y}}\cup id_{E_K}$ is well defined. To simplify notation we write $f$ in place of $H|_{\mathring{Y}}\cup id_{E_K}.$ 
\end{definition}
Analogous to \cite[Proposition 6.2]{conway2025unknotting} which treats non-orientable surfaces in $S^4$ with $\pi_1(S^4\setminus F)\cong \Z_2$, we state our main theorem of this section:
\begin{theorem}\label{SpinU}
    Let $F_1,F_2\subseteq X$ be two homologous surfaces of genus $g$, with the same non-empty boundary $K\subseteq S^3,$ such that $X_{F_i}\; (\operatorname{for}\;i=1,2)$ are simply-connected and spin. Let $\mathfrak{s}^{\partial}_2\in \Spin(\partial X_{F_2})$ be the restriction of $\mathfrak{s}\in \Spin{(X_{F_2})}$ to the boundary. Given a bundle isomorphism $H:(\overline{\nu}F_1,F_1)\to(\overline{\nu}F_2,F_2)$ then
    \begin{enumerate}
        \item When $X$ is spin, there exists another bundle isomorphism $H':(\overline{\nu}F_1,F_1)\to(\overline{\nu}F_2,F_2)$ such that $H'|_{F_1}=H|_{F_1},$ and $f^{\ast}\s^{\partial}_2:=(H'|_{\mathring{Y_1}}\cup id_{E_K})^{\ast}\mathfrak{s}^{\partial}_2$ extends over $X_{F_1}.$
        \item When $F_1,F_2$ are characteristic, if $H|_{F_1}$ preserves the Rokhlin quadratic form (see Definition~\ref{rokh}) then $f^{\ast}\s^{\partial}_2:=(H|_{\mathring{Y_1}}\cup id_{E_K})^{\ast}\mathfrak{s}^{\partial}_2$ extends over $X_{F_1}.$
    \end{enumerate}
\end{theorem}
We record a corollary of our main result to be used later in the proof of Theorem \ref{mainthm}.
\begin{corollary} 
    If $X_{F_1}$ and $X_{F_2}$ are simply-connected and spin, then $X_{F_1}\cup_f-X_{F_2}$ is a simply-connected, closed, spin 4-manifold.
\end{corollary}
The proof of Theorem \ref{SpinU} will decompose into two cases by the following lemma:
\begin{lemma}\cite[Lemma 7.18]{friedl2025foundations} Given $F\subset X$ a properly embedded surface, then $X_F$ is spin if and only if\begin{itemize}\item[1.] $X$ is spin,\item[2.] or $F$ is characteristic.\end{itemize}\end{lemma}
Now we give an outline of the upcoming proof of Theorem \ref{SpinU}:

\emph{For the spin case:}
\begin{enumerate}
\item We first restrict the unique spin structure, $\mathfrak{s}$ on $X$ to the tubular neighborhood of both surfaces, $\mathfrak{s}^{\overline\nu F}_{i}\in \Spin(\overline\nu F_i),\;(\operatorname{for}\;i=1,2).$
\item It may not be the case that the bundle isomorphism $H,$ obeys the equality $H^{\ast}\mathfrak{s}^{\overline\nu F}_{2}=\mathfrak{s}^{\overline\nu F}_{1}.$ So, we modify $H$ without altering the 0-section.
\item Since $H^1(\overline{\nu}F_1;\Z_2)$ acts transitively on $\Spin(\overline {\nu}F_1)$ we can find an element $x\in H^1(\overline{\nu}F_1)$ that reduces mod 2 to $\tilde{x}\in H^1(\overline\nu F_1;\Z_2)$ such that $\tilde{x}\cdot\mathfrak{s}^{\overline\nu F}_{2}=\mathfrak{s}^{\overline\nu F}_{1}.$ By Brown representability we have a one to one correspondence between first cohomology classes of $H^1(F;\Z)$ and free homotopy classes $[F_1,S^1].$ So we first choose a suitable map $F\to S^1.$
\item We consider the bundle isomorphism $R:(\overline{\nu}F,F_1)\to (\overline{\nu}F,F_1) $ that can be described as taking a fixed $S^1-$action on $D^2$ and extending fiberwise using the map $F\to S^1$ (see Section \ref{pf1}). 
\item We deduce that $H':=(H\circ R)$ is a bundle isomorphism that obeys the equality $H'^{\ast}\mathfrak{s}^{\overline\nu F}_{2}=\mathfrak{s}^{\overline\nu F}_{1}.$ Thus finally concluding with $f:=H'|_{\mathring{Y}}\cup id_{E_K}$ has the desired property that $f^{\ast}\s_2$ extends over $X_{F_1}.$ 
\end{enumerate}

\emph{In the $F$ characteristic case:}
\begin{enumerate}
    \item After defining the Rokhlin and Kirby-Taylor quadratic forms (see Definitions \ref{rokh} and \ref{defKT}) Corollary~\ref{OrderSPin} ensures that for a particular spin structure $\mathfrak{q}\in \Spin{\partial X_F},$ does not extend as a spin structure $\s_{\mathfrak{q}}\in \Spin(X_F)$ unless the Rokhlin and Kirby-Taylor quadratic forms with respect to $\mathfrak{q}$ agree.

    \item Next we use \cite[Proposition 6.10]{conway2025unknotting} to distinguish the elements in $\{\mathfrak{s}\in \Spin(\mathring{Y}):q_{KT}(\mathfrak{s})=q_{FK}\text{ on } F\}$ by whether $\mathfrak{s}$ restricted to an $S^1-$fiber extends over $D^2.$ 
    \item Lastly, for $\mathfrak{s}_2\in \Spin(\partial X_{F_2})$ restricted to $\mathfrak{s}^{\mathring{Y}}_1\in\Spin(\mathring{Y}),$ we show $q_{KT}((H|_{\mathring{Y}_1})^{\ast}\mathfrak{s}^{\mathring{Y}}_2)=q_{FK}\text{ on }F_1.$ Then we show $(H|_{\mathring{Y}_1})^{\ast}\mathfrak{s}^{\mathring{Y}}_2$ restricted to an $S^1-$fiber does not extend over the disc by the fact $H$ is a bundle isomorphism. Thus finally concluding with $f:=H|_{\mathring{Y}}\cup id_{E_K}$ has the desired property that $f^{\ast}\s_2$ extends over $X_{F_1}.$ 
\end{enumerate}
\begin{remark}
    We mention that much of this argument, inspired by \cite[Section 4]{Boyer1993}, is adapting work done in \cite[Section 6]{conway2025unknotting} to the orientable setting.
\end{remark}
\subsection{Kirby-Taylor quadratic refinements} 
We recall the definition of a quadratic refinement on an abstract surface given by \cite{kirby_taylor_pin} (see also Section 6 of \cite{conway2025unknotting}).
\begin{notation}
    In this section we use the notation $\nu _{A\subseteq B}$ for the normal vector bundle of a submanifold $A\subseteq B$ as opposed to $\nu A\subset B,$ which is an open tubular neighborhood. 
\end{notation}

    Let  $F$ be an abstract orientable surface, and $\zeta:=(E,F,\pi)$ be a trivial line bundle over $F.$ 
    Note $TF\oplus \zeta$ admits a spin structure as both $TF$ and $\zeta$ are spin thus by the Whitney sum formula $TF\oplus \zeta$ is spin. Fix such a spin structure $\mathfrak{s}$ on $TF\oplus \zeta$ and let $\gamma\subseteq F$ be an embedded circle. Make a choice of orientation on $\gamma.$ This choice determines a framing $T\gamma\cong \gamma \times \R.$

Now we describe two framings on the rank three vector bundle $TF|_{\gamma}\oplus \nu_{F\subset E}|_{\gamma}. $

\begin{itemize}
    \item 
Firstly, $\mathfrak{s}$ completely determines a framing given by the following identifications:
$$\mathcal{F}:TF|_{\gamma}\oplus \nu_{F\subseteq E}|_{\gamma}\cong TF|_{\gamma}\oplus \zeta|_{\gamma}\cong (TF\oplus \zeta)|_{\gamma}\cong \gamma \times \R^3$$
\item Secondly, we consider the framing $\mathcal{G}_f$ obtained by combining the orientation-induced framing $T\gamma\cong \gamma \times
\R$ with some additional choice of framing $f:\nu_{\gamma \subseteq E}\cong \gamma \times \R^2$ in the following way:
$$\mathcal{G}_f:TF|_{\gamma}\oplus \nu_{F\subseteq E}|_{\gamma}\cong T{\gamma}\oplus \nu_{\gamma \subseteq F}\oplus \nu_{F\subseteq E}|_{\gamma}\cong T\gamma \oplus \nu_{\gamma \subseteq E}\cong \gamma \times \R^3.$$
    \end{itemize}
\begin{definition}[\cite{kirby_taylor_pin}\cite{conway2025unknotting}]
A framing $f:\nu_{\gamma \subseteq E}\cong \gamma \times \R^2,$ is called \emph{odd} with respect to a fixed spin structure $\mathfrak{s},$ if $\mathcal{G}_f$ determines the same orientation on $\nu_{\gamma\subseteq E}$ as $\mathcal{F},$ but the framings disagree up to homotopy.   
\end{definition}

   \begin{definition}[\cite{kirby_taylor_pin}] \label{defKT}
    The \emph{Kirby-Taylor quadratic form} $q_{KT}(\mathfrak{s}):H_1(F;\Z_2)\to \Z_2$ is defined as $$q_{KT}(\mathfrak{s})(\gamma):=  \#
\left\{\begin{array}{cc}
 \text{right-handed full-twists made by }\nu_{\gamma\subseteq F}\subseteq\nu_{\gamma \subseteq E}
\\
 \text{as it completes a full rotation around}~\gamma \text{ in the positive direction}
 \end{array}\right\} \in\Z_2.$$ The Kirby-Taylor quadratic refinement depends neither on the choice of orientation on $\gamma,$ nor the particular choice of odd framing; see \cite[p. 209]{kirby_taylor_pin}. In practice we will take $\zeta:=\nu_{s(F)\subset \mathring{Y}}$ for $s:F\to \mathring{Y}$ a section. 
\end{definition}
\begin{remark}
    For the reader comparing with \cite{kirby_taylor_pin}, \cite{conway2025unknotting}, we note that both papers allow $F$ to be non-orientable and hence originally the Kirby-Taylor quadratic form was defined as a $\Z_4$ quadratic refinement over $\Z_2.$ Orientable surfaces have alternating $\Z_2$ intersection forms from which one deduces that $\Z_4-$quadratic refinements of the forms are $2\Z_4=\Z_2$ valued.
\end{remark}
We restate here Proposition 6.10 of \cite{conway2025unknotting} for convenience to the reader.
\begin{proposition}\cite[Proposition 6.10]{conway2025unknotting}
    Let $F,F_1,\text{ and } F_2$ be orientable surfaces with connected nonempty boundary.
    \begin{enumerate}
        \item Let $\xi$ be a rank-2 vector bundle over $F$ with orientable total space. Write $\mathring{Y}$ for the total space of the associated $S^1-$bundle $S(\xi).$ Fix a choice of section $s:F\to \mathring{Y},$ and a $S^1-$fibre $\mu\subset\mathring{Y}$ There is a bijection of sets \[
\begin{array}{rcl}
\Psi_{s,\mu}\colon\Spin(\mathring{Y})&\xrightarrow{1:1}& \Spin(T\mathring{Y}|_\mu)\times\left\{\begin{array}{cc}\text{$\Z_2$-quadratic refinements of}\\ \text{$(H_1(s(F);\Z_2),Q_{s(F)})$}\end{array}\right\}\\
\s&\mapsto& (\s|_{\mu},q_{KT}(\s)).
\end{array}
\]
To define $q_{KT}(\s)$, we take $\zeta := \nu_{s(F) \subseteq \mathring{Y}}$ and use the spin structure on $\mathring{Y}$ to induce a spin structure on $TF \oplus \nu_{s(F) \subseteq \mathring{Y}} \cong T\mathring{Y}|_{s(F)}$.

\item Let~$\Phi\colon\xi_1\cong \xi_2$ be an isomorphism of orientable 2-plane vector bundles with orientable total space, covering a homeomorphism~$\varphi \colon F_1 \to F_2$.
Write~$\mathring{Y}_i$ for the total space of the associated~$S^1$-bundles~$S(\xi_i)$. Fix a choice of section $s_1:F_1\to \mathring{Y}_1|_{F_1}$ and a choice of~$S^1$-fiber~$\mu_1 \subseteq \mathring{Y}_1$. Write $s_2=\Phi \circ s_1\circ \varphi^{-1}$ and $\mu_2=\Phi(\mu_1).$

Then there is a commutative diagram
\[
\begin{tikzcd}
\Spin(\mathring{Y}_2)\ar[rr, "\Psi_{\s_2,\mu_2}"]\ar[d, "{\Phi^*}"]
&&\Spin(T\mathring{Y}_2|_{\mu_2})\times \left\{\parbox{4.5cm}{\centering~$\Z_2$-quadratic refinements of\\~$(H_1(s(F_2);\Z_2),Q_{s(F_2)})$}\right\} \ar[d, "{\widehat{\Phi}}"] \\
\Spin(\mathring{Y}_1)\ar[rr, "\Psi_{\mathfrak{s}_1,\mu_1}"]
&&\Spin(T\mathring{Y}_1|_{\mu_1})\times \left\{\parbox{4.5cm}{\centering~$\Z_2$-quadratic refinements of\\~$(H_1(s(F_1);\Z_2),Q_{s(F_1)})$}\right\},
\end{tikzcd}
\]
where the~$\Psi_{s_i,\mu_i}$ are the bijections of sets from the first item, and

\[
{\widehat{\Phi}}(\mathfrak{t}, q):= ((\Phi|_{\mu_2})^*\mathfrak{t},  q \circ \Phi_*).
\]

    \end{enumerate}\label{COP}
\end{proposition}
\cite[Proposition 6.10]{conway2025unknotting} originally is only stated for non-orientable surfaces. The two changes needed to adapt their proof are as follows:
\begin{enumerate}
    \item Consider a standard cell structure for $F,$ with basis loops $\gamma_1,\cdots,\gamma_{2g}.$ Let $A\subseteq \mathring Y$ be a connected one complex comprised of a single 0-cell, the collection of loops $s(\gamma_i),$ and the $S^1$ fiber $\mu_1.$ We obtain a bijection from $\operatorname{Spin}(\mathring{Y})\to \operatorname{Spin}(T\mathring{Y}|_{A}),$ by taking the inclusion induced map on first homology and deducing by Lemma \ref{HomologyY} $H_1(A;\Z_2)\xrightarrow[]{\cong}H_1(\mathring{Y};\Z_2).$
    \item After applying \cite[Lemma 1.7]{kirby_taylor_pin} (see proof of \cite[Theorem 6.7]{conway2025unknotting})one arrives at a bijection $$\Psi_{s,\mu}:\operatorname{Spin}(Ts(F)\oplus\nu_{s(F)\subseteq\mathring{Y}})\xrightarrow{1:1}\left\{\parbox{4.5cm}{\centering~$\Z_4$-quadratic refinements of\\~$(H_1(s(F);\Z_2),Q_{s(F)})$}\right\}$$
\end{enumerate}
 which refines to $\Z_2$-quadratic refinements of $(H_1(s(F);\Z_2),Q_{s(F)})$ in the orientable case. The rest of the proof has no dependence on orientability of the starting surface.

\subsection{The Rokhlin quadratic refinement} Let $F\subseteq X$ be a properly embedded orientable characteristic surface with connected nonempty boundary. We also assume that $\pi_1(X) = 1.$ Let us review the Rokhlin quadratic refinement as given in \cite{freedman_kirby_rochlin} (see also \cite{klug2021relativeversionrochlinstheorem}):

Let $\gamma$ be an embedded circle in the interior of $F$ and let $D$ be an immersed disk bounded by $\gamma$ in $X$ such that $D$ is transverse to $F$ at $\gamma.$ Moreover, we assume the interior of $D$ meets $F$ transversely away from $\gamma.$ Choose an orientation on $\gamma$ and thus a positive tangent direction. Since the oriented normal bundle $\nu_{D\subseteq X}$ is trivial, choose a trivialization $\nu_{D\subseteq X}\cong D\times \R^2.$ Therefore, choose the framing on the normal bundle, $\nu_{D\subset X}|_{\gamma},$ for $D$ a properly embedded disk transverse to $F$ as above, such that the frame
\begin{center}
    (outward facing normal to $\gamma$ in $D)\times$(positive tangent to $\gamma\times e_1\times e_2$
\end{center}
 agrees with the ambient orientation of $X$ at all points $p\in \nu(\gamma).$
Define $\mathcal{O}(D)$ to be the number of right-handed full-twists made by this 1-dimensional subbundle $\nu_{\gamma\subseteq F}\subset\nu_{D\subseteq X}|_{\gamma}$ as it completes a full rotation around $\gamma$ in the positive direction. Now we define the Rokhlin quadratic form.

\begin{definition}[\cite{rokhlin_gudkov}] \label{rokh}   
The \emph{Rokhlin quadratic} form for a characteristic surface $F\subset X$ is $$q_{FK}:H_1(F;\Z_2)\to \Z_2 $$  $$[\gamma]\mapsto \mathcal{O}(D)+D\cdot F\pmod2.$$ 
\end{definition}

\subsection{The exterior Rokhlin form} In this section, we define the exterior Rokhlin form which will let us translate between the classical Rokhlin and Kirby-Taylor quadratic forms i.e., we define a way to compute the Rokhlin quadratic refinement of $F$ using disks whose interiors lie entirely in the complement of $F.$

Let $F\subseteq X$ be a surface with simply-connected complement. Pick a section $s:F\to \mathring{Y},$ and observe that for every $\gamma\subset F,\; s(\gamma)\subset X_F$ is null-homotopic in the exterior. Generally, we say a section is \emph{nice} if the composition $$H_1(F)\xrightarrow[]{s_{\ast}}H_1(\mathring{Y})\xrightarrow[]{i_{\mathring Y}}H_1(\partial X_F)\xrightarrow[]{i_{\partial X_F}}H_1(X_F)$$ is identically zero. So now we have existence of properly immersed disks $D\looparrowright X_F$ that are transverse to $s(F)$ along $\partial D=s(\gamma).$

\begin{definition}
    Fix nice section $s:F\to \mathring{Y}.$ Let $D\looparrowright X_F$ be a properly immersed disk transverse to $s(F)$ along $s(\gamma),$ for an embedded circle $\gamma\subset F.$ We now choose the framing of $\nu_{D\subseteq X_F}|_{\\\gamma}$ in the same way as before: 
     Fix an orientation on $\gamma$ and choose a framing on the normal bundle, $\nu _{D\subseteq X_F}|_{\gamma},$ for $D$ a properly immersed disk transverse to $F$ as above, such that the frame
\begin{center}
    (outward facing normal to $\gamma$ in $D)\times$(positive tangent to $\gamma)\times e_1\times e_2$
\end{center}
agrees with the ambient orientation of $X_F$ at all points $p\in \gamma.$ The \emph{exterior Rokhlin quadratic} form is defined as:
$$\hat{q}_{FK}(\gamma):=\mathcal{O}(D).$$
As there is a basis of $H_1(s(F);\Z_2)$ by oriented embedded loops, this extends linearly to a function $\hat{q}_{FK}.$ The function is well-defined by a similar argument to why the Rokhlin function is well-defined.
\end{definition}
 Now we relate the Rokhlin and exterior Rokhlin forms by appealing to the proof of \cite[Proposition 6.15]{conway2025unknotting} and the exterior Rokhlin to the Kirby-Taylor quadratic form by appealing to the proof of  \cite[Corollary 6.19]{conway2025unknotting} since both arguments are identical in the orientable case.
\begin{proposition}
    Let $F\subset X$ be a characteristic surface such that $\pi_1(X\setminus F)=1,$ and fix a nice section $s$ of $\mathring{Y}.$ The exterior Rokhlin quadratic form is a $\Z_2-$quadratic refinement of the $\Z_2-$intersection form $(H_1(s(F);\Z_2),Q_{s(F)});$ the isomorphism $s_{\ast}:H_1(F;\Z_2)\to H_1(s(F);\Z_2)$ is an isometry of the quadratic forms 
    $$(Q_F,q_{FK})\xrightarrow[]{\cong}(Q_{s(F)},\hat{q}_{FK}).$$\label{ext}
\end{proposition}

\begin{proposition}\label{lem:optimistic}
Let $F\subset X$ be a surface such that $\pi_1(X\setminus F)=1.$  Choose a spin structure~$\mathfrak{s}$ for~$\partial X_F$ that extends to a spin structure on~$X_F$ and choose a section $s:F\to \mathring{Y}.$
Then there is an equality of the ~$\Z_2$-quadratic forms
\[(H_1(s(F);\Z_2),Q_{s(F)},q_{FK}) = ((H_1( F;\Z_2),Q_{F},\hat{q}_{KT}(\mathfrak{s}))\]
\end{proposition}

Now we give a bound on the number of spin structures over $\mathring{Y}$ that extend over $X_F.$
\begin{remark}
    Given embedded submanifold $A\subset B,$ we use the notation $\Spin(B,A)$ to denote the set $\{\mathfrak{s}\in \Spin(A):\mathfrak{s}\text{ extends over } B\}.$
\end{remark}
\begin{corollary}\label{OrderSPin} 
Fix a section $s:F\to \mathring{Y}.$ The set $\Spin(X_F,\mathring{Y}):=\{\mathfrak{s}\in \Spin(\mathring {Y}):\mathfrak{s}\text{ extends over } X_F\},$ is contained in $\{\mathfrak{s}\in \Spin(\mathring{Y}):q_{KT}(\mathfrak{s})=\hat{q}_{FK}\text{ on }H_1(F;\Z_2)\}.$ Moreover when $X_F$ is spin,  $$1\leq |\Spin(X_F,\mathring{Y})|\leq |\{\mathfrak{s}\in \Spin(\mathring{Y}):q_{KT}(\mathfrak{s})=\hat{q}_{FK}\text{ on }F\}|= 2.  $$
\end{corollary}
\begin{proof}
The containment follows from Proposition~\ref{lem:optimistic}: any spin structure on $\mathring Y$ extending over $X_F$ has Kirby-Taylor form equal to the exterior Rokhlin form.

Now from \cite[Proposition 6.10 (1)]{conway2025unknotting} (see Proposition \ref{COP}) the following map is a bijection  \[
\begin{array}{rcl}
\Psi_{s,\mu}\colon\Spin(\mathring{Y})&\xrightarrow{1:1}& \Spin(T\mathring{Y}|_\mu)\times\left\{\begin{array}{cc}\text{$\Z_2$-quadratic refinements of}\\ \text{$(H_1(F;\Z_2),Q_{F})$}\end{array}\right\}\\
\s&\mapsto& (\s|_{\mu},q_{KT}(\s)).
\end{array}
\] For the right-hand inequality, since we restricted to $\mathfrak{s}\in \Spin(X_F,\mathring{Y}),$ this tells us there is a unique $\Z_2-$refinement, and there are at most two spin structures $\mathfrak{s}\in \Spin(X_F,\mathring{Y}),$ with $q_{KT}(\mathfrak{s})=\hat{q}_{FK}$ distinguished by the value of the restriction of $\mathfrak{s}$ to the meridian of $F.$ Finally, $X_F$ being spin gives the left-hand side inequality. 
\end{proof}

\subsection{The Proof of Theorem~\ref{SpinU}}\label{pf1}
Now we prove the main result of this section. Let $F_1,F_2\subseteq X$ be two homologous, genus $g$ surfaces, with the same non-empty boundary $K\subset S^3,$ such that $\pi_1(X_{F_i})=1.$ 

\subsubsection{The proof when X is Spin} We adapt an argument from \cite[Proof of Theorem F and part (i) of Theorem G]{Boyer1993} to verify the result. Suppose $X$ is spin. Let $\mathfrak{s}\in \Spin(X)$ be the unique spin structure on $X,$ and  $\mathfrak{s}_i\in\Spin(X_{F_i})$ be the restrictions of $\mathfrak{s}.$ Moreover, restrict $\mathfrak{s}$ to $\mathfrak{s}^{\overline{\nu}F}_i\in \Spin(\overline{\nu}F_i).$

It need not be the case that $H^{\ast}\mathfrak{s}^{\overline{\nu}F}_2=\mathfrak{s}^{\overline{\nu}F}_1,$ but this may be corrected without altering $H|_{F_1}.$ Recall $H^{1}(\overline{\nu}F_1;\Z_2)$ acts freely and transitively on $\Spin(\overline{\nu}F_1),$ so there exists an element $y\in H^{1}(\overline{\nu}F_1;\Z_2) $ such that $y\cdot H^{\ast}\mathfrak{s}^{\overline{\nu}F}_2=\mathfrak{s}^{\overline{\nu}F}_1.$ Let $p:\overline{\nu}F_1\to F_1$ be the standard projection map and fix a class $\tilde{y}\in H^1(F_1;\Z)$ such that $p^{\ast}\tilde{y}=y\;(mod\;2).$ 

Choose a map $\gamma:F_1\to S^1$ representing $\tilde y\in H^1(F_1;\mathbb Z)$ under the standard identification $H^1(F_1;\mathbb Z)\cong [F_1,S^1]$.
Define $R:(\overline\nu{F_1},F_1)\to (\overline\nu{F_1},F_1),$ a bundle isomorphism, fiberwise by
\[
R(z)=\gamma(p(z))\cdot z,
\]
where $\gamma(p(z))\in S^1$ acts by rotation on the disk fiber of $\overline\nu F_1$ over $p(z)$.

Note that $R|_{\mathring{Y}_1}$ and $R|_{F_1}$ the 0-section both are equivalent to the identity on $\mathring{Y}_1$ and $F_1$ respectively.  After precomposing $H$ with $R$ we have the following equalities:$$(H\circ R)^{\ast}\mathfrak{s}^{\overline{\nu}F}_2=R^{\ast}(H^{\ast}\mathfrak{s}^{\overline{\nu}F}_2)=y\cdot (H^{\ast}\mathfrak{s}^{\overline{\nu}F}_2)=\mathfrak{s}^{\overline{\nu}F}_1.$$ Now define $H'=(H\circ R),$ and $f:=H'|_{\mathring{Y}_1}\cup id_{E_K.}$ Since both $\mathfrak{s}^{\partial}_1$ and $f^{\ast}\mathfrak{s}^{\partial}_2$ arise as restrictions of the spin structure $\mathfrak{s}$ on $X$, their restrictions to $E_K$ agree. Moreover, $H'$ was chosen so that $(H')^{\ast}\mathfrak{s}^{\overline{\nu}F}_2=\mathfrak{s}^{\overline{\nu}F}_1$, and hence the induced spin structures agree on $\mathring{Y}$. Therefore $\mathfrak{s}^{\partial}_1$ and $f^{\ast}\mathfrak{s}^{\partial}_2$ agree on both pieces of the decomposition $\partial X_{F_1}=E_K\cup \mathring{Y}_1$, and so
\[
f^{\ast}\mathfrak{s}^{\partial}_2=\mathfrak{s}^{\partial}_1.
\]

As $\mathfrak{s}^{\partial}_i$ was defined as the restriction of $\mathfrak{s}$ to $\partial X_{F_i},$ we conclude that $\mathfrak{s}^{\partial}_1$ extends over $X_{F_1},$ completing the proof of the case when $X$ is spin.

\subsubsection{The proof when $F$ is characteristic}
Since $F_i\subset X$ are characteristic and $\pi_1(X_{F_i})=1$ that implies $X$ is nonspin. Additionally, suppose $H|_{F_1}$ preserves the Rokhlin quadratic form. Consider the unique spin structures $\mathfrak{s}_i\in \Spin(X_{F_i})$ for $i=1,2.$ Restrict to $\mathfrak{s}^{\mathring{Y}}_i\in\Spin({\mathring{Y}_i}).$ Observe that $\mathfrak{s}^{\mathring{Y}}_i\notin\Spin(\overline{\nu}F_i,\mathring{Y}_i),$ as if so then $X:=X_{F_i}\cup\overline{\nu}F_i$ would be spin, a contradiction. 
Now consider the following claim.
\begin{claim}
Let $F\subset X$ be a surface with $\pi_1(X_F)=1,$ and $s:F\to \mathring{Y}$ be a section. For $\mathfrak{q}\in \Spin(\mathring{Y}),$ if the Kirby-Taylor quadratic form associated to $\mathfrak{q,\;}q_{KT}(\mathfrak{q})$ agrees with the Rokhlin quadratic form associated to $F\subset X,\;\hat{q}_{FK},$ and $\mathfrak{q}$ restricted to an $S^1-$fiber of $\mathring{Y}$ does not extend over the disk, then $\mathfrak{q}\in \Spin(X_F,\mathring{Y}).$
\end{claim}
\begin{proof}[Proof of Claim]
    Consider the set $\{\mathfrak{s}\in \Spin(\mathring{Y}):q_{KT}(\mathfrak{s})=\hat{q}_{FK}\text{ on }H_1(F;\Z_2)\}.$ By Corollary~\ref{OrderSPin} the set has at most 2 elements, and contains $\Spin(X_F,\mathring{Y}).$ By \cite{conway2025unknotting}[Proposition 6.10(i)] (see Proposition \ref{COP}) we have that if $\{\mathfrak{s}\in \Spin(\mathring{Y}):q_{KT}(\mathfrak{s})=\hat{q}_{FK}\text{ on }H_1(F;\Z_2)\}$ realizes two elements then they are distinguished by their restriction to an $S^1-$fiber of $\mathring{Y},$ i.e., whether $\mathfrak{s}|_{\mu}$ extends as a spin structure over the disk. Since $X$ is nonspin, but $X_F$ is spin it must be the case that $\mathfrak{s}|_{\mu}$ does not extend over the disk, hence if $q_{KT}(\mathfrak{s})=\hat{q}_{FK}$ and $\mathfrak{s}|_{\mu}$ does not extend over the disk then $\mathfrak{s}\in\{\mathfrak{s}\in \Spin(\mathring{Y}):q_{KT}(\mathfrak{s})=\hat{q}_{FK}\text{ on }H_1(F;\Z_2)\} $ belongs to the subset $\Spin(X_F,\mathring{Y}).$
\end{proof}
\begin{proof}[Proof of Case 2] 
Fix $\mu_2\subset \mathring{Y}_2,$ and define $\mu_1:=(H|_{\mathring{Y}_1})^{-1}(\mu_2),$ and $\mathfrak{g}\in \Spin(\mathring{Y}_1)$ as $\mathfrak{g:=}(H|_{\mathring{Y}_1})^{\ast}\mathfrak{s}^{\mathring{Y}}_2.$ Additionally, fix a section $s_2:F_2\to\mathring{Y}_2,$ and set $s_1:=H^{-1}\circ s_2 \circ H|_{F_1}:F_1\to \mathring{Y}_1.$ Note the equality $H\circ s_1=s_2\circ H|_{F_1}.$

Since $\mathfrak{s}^{\mathring Y}_2 \notin \Spin(\overline\nu F_2, \mathring Y_2)$, its restriction $\mathfrak{s}^{\mathring Y}_2\vert{\mu_2}$ to a fiber does not extend over the disk. Because $H$ is a bundle isomorphism, the analogous statement holds for $\mathfrak{g}\vert{\mu_1}$.

To conclude the proof assuming the claim we need to show $q_{KT}(\mathfrak{g})=\hat{q}_{FK_1}$ for $\hat{q}_{FK_i}$ the exterior Rokhlin quadratic form associated to $F_i\subset X\;(\operatorname{for}\;i=1,2).$ Now we invoke the following equalities: 
\begin{align*}
q_{KT}(\mathfrak{g})&=q_{KT}(H^{\ast}\mathfrak{s}^{\mathring{Y}}_2)\text{ by definition}\\&=q_{KT}(\mathfrak{s}^{\mathring{Y}}_2)\circ H_{\ast}\text{ by \cite[Proposition 6.10(2)]{conway2025unknotting}}\text{ and } H\circ s_1=s_2\circ H|_{F_1}\\&=\hat{q}_{FK_2}\circ H_{\ast}\text{ by Proposition \ref{lem:optimistic} since }\mathfrak{s}^{\mathring{Y}}_2\text{ extends over }X_{F_2} 
\\&=\hat{q}_{FK_1} \text{ by assumption and }H\circ s_1=s_2\circ H|_{F_1}
\end{align*}

Under the boundary decomposition $\partial X_F=E_K\cup \mathring Y$, the $S^1$-fiber of $\mathring Y$ is identified with the meridian of $E_K$. Since $E_K\subset S^3$ inherits the restriction of the unique spin structure on $S^3$, the restriction of the boundary spin structure to $E_K$ is determined by its value on this meridian. Hence agreement on the fiber implies agreement on the $E_K$-piece.
 
Hence we have again the identity for $f:=((H|_{\mathring{Y}_1})\cup id_{E_K})^{\ast}\mathfrak{s}_2=\mathfrak{s}_1.$
\end{proof}

\subsubsection{Existence of special bundle isomorphisms}
We record an existence result for bundle isomorphisms that preserve the Rokhlin quadratic form. 
\begin{proposition}
    
\label{cor:realiseisometry}
Let ~$F_1,F_2\subseteq X$ be genus $g$ surfaces with boundary $K\subseteq S^3$ and equal relative Euler number $e$ and $\pi_1(X_{F_i})=1$, then there exists a homeomorphism~$f\colon F_1\to F_2$, restricting to the identity map on $K$, and inducing an isometry between the Rokhlin forms
\[
f_*\colon (H_1(F_1;\Z_2),Q_{F_1},q_{FK_1})\xrightarrow{\cong} (H_1(F_2;\Z_2),Q_{F_2},q_{FK_2}).
\]
Thus there exists bundle isomorphism $H:(\overline{\nu}F_1,F_1)\to (\overline{\nu}F_2,F_2),$ such that $H$ restricts to the identity map on $K$, and induces an isometry between the Rokhlin forms.\label{exist}

\end{proposition}

\begin{proof}
$F_1$ and $F_2$ have the same genus, relative Euler numbers and connected boundary $K,$ therefore, there exists an isometry of their nonsingular $\Z_2-$intersection forms $(H_1(F_1;\Z_2),Q_{F_1})\cong (H_1(F_2;\Z_2),Q_{F_2})$ and \cite[Theorem 2]{klug2021relativeversionrochlinstheorem} implies that their Rokhlin forms have the same Arf invariants. By \cite[Theorem III.1.12]{Browder1972Surgery}, nonsingularity of the intersection form there exists some isometry of their $\Z_2$ quadratic forms $(H_1(F_1;\Z_2),Q_{F_1},q_{FK_1})\cong (H_1(F_2;\Z_2),Q_{F_2},q_{FK_2})$ and by \cite[Theorem 6.4]{FarbMargalit2012} this isometry is induced by a homeomorphism between the compact, same genus orientable surfaces with connected non-empty boundary that fixes said boundary point-wise. Since any homeomorphism of a surface lifts to a bundle isomorphism of the trivial $D^2$-bundle, by extending fiberwise by the identity this completes the proof. 
\end{proof}

\section{Proof of the main theorem.}\label{end}
The goal of this section is to prove Theorem \ref{mainthm}. Unless stated otherwise, for the rest of this section fix $X$ to be a simply-connected 4-manifold with boundary homeomorphic to $S^3,$ $K,$ a knot in $S^3,$ and embedded, homologous, genus $g$ surfaces $F_1$ and $F_2$ with boundary $K,$ such that $\pi_1(X_{F_i})=1$ for $i=1,2.$
\subsection{Strategy of the proof.} Recall the goal is to construct a homeomorphism of pairs $(X,F_1)\to(X,F_2)$ isotopic to the identity. The strategy is to pair a bundle isomorphism $$H:(\overline\nu F_1,F_1)\to  (\overline\nu F_2,F_2),\text{where } H|_{\nu K}=id_{\nu K},$$ with an isometry $\Lambda:Q_{X_{F_1}}\cong Q_{X_{F_2}}$ such that $(H|_{\mathring{Y}} \cup\; id_{E_K},\Lambda)$ is compatible (see Definition~\ref{pair}). Setting $f:=(H|_{\mathring{Y}_1})\cup \;id_{E_K},$ we show such a pair $(f,\Lambda)$ can be realized as the induced maps of a homeomorphism $F:X_{F_1}\to X_{F_2}.$ Finally we show the homeomorphism $\hat{F}:=F\cup H:(X,F_1)\to (X,F_2)$ is isotopic to the identity rel. boundary on $X.$ 

The main results of this section are the following:

\begin{proposition}
     Given a bundle isomorphism, $H:(\overline\nu F_1,F_1)\to  (\overline\nu F_2,F_2)$ fix $f:=H|_{\mathring{Y}_1}\cup id_{E_K}.$  There exists an isometry $\Lambda:H_2(X_{F_1})\to H_2(X_{F_2})$ such that diagram \ref{4.2} commutes and $\Lambda$ induces $f_{\ast}:H_i(\partial X_{F_1})\to H_i(\partial X_{F_2})$ for $i=1,2.$ \label{main1}
\end{proposition}

\begin{customthm}{Lemma 4.6}
    Assume $X_{F_i}$ is non-spin $\operatorname{for}\;i=1,2.$ Fix a compatible pair $(f,\Lambda)$ making diagram \ref{4.2} commute. There exists an isometry  $\Lambda':H_2(X_{F_1})\to H_2(X_{F_2})$ such that $(f,\Lambda')$ remains compatible, makes diagram \ref{4.2} commute, and is realized by a homeomorphism $F:X_{F_1}\to X_{F_2}.$
\end{customthm}

\begin{proposition}\label{main3}
   If a homeomorphism $F:X_{F_1}\to X_{F_2}$ restricted to the boundary agrees with $f:=H|_{\mathring{Y}}\cup id_{E_K}$ and $(f,\Lambda):=(F|_{\partial},F_{\ast})$ makes diagram \ref{4.2} commute then $\hat{F}=F\cup H:X\to X$ sends $F_1\text{ to } F_2$ and is isotopic  rel. boundary to the identity. 
\end{proposition}

Some preparation is required for the proof of these statements, which is the goal of this section, but assuming the statements we prove Theorem \ref{mainthm}.

\begin{customthm}{Theorem 1.1}
    Let $K$ be a knot in $S^3.$ Any two homologous orientable genus $g$ surfaces $F_1,F_2\subset X,$ both with boundary $K$ and $\pi_1(X\setminus F_i)=1$ for $i=1,2,$ are topologically ambiently isotopic rel. boundary. 
\end{customthm}

\begin{proof}[Proof of Theorem 1.1]

Begin by constructing a homeomorphism over the surface exteriors:

\textbf{In the case of spin exteriors}, start with a bundle isomorphism $H:(\overline{\nu} F_1,F_1)\to (\overline{\nu} F_2,F_2).$ Proposition~\ref{exist} along with Theorem \ref{SpinU} guarantees existence of $H'$ such that the homeomorphism $f:=H'|_{\mathring{Y}}\cup id_{E_K}:\partial X_{F_1}\to \partial X_{F_2}$ has the property that $X_{F_1}\cup_f -X_{F_2}$ is spin. Now apply Proposition \ref{main1} to obtain an isometry compatible with $f$ and making diagram 4.2 commute. By \cite[Proposition 0.8]{boyer1986simply}, $(f,\Lambda)$ is induced by a homeomorphism $F:X_{F_1}\to X_{F_2}$  if and only if the union $X_{F_1}\cup_fX_{F_2}$ is spin.

\textbf{For the case of non-spin exteriors},

start with any bundle isomorphism $H:(\overline\nu F_1,F_1)\to  (\overline\nu F_2,F_2),$ and consider $f:=H|_{\mathring{Y}}\cup id_{E_K}.$ By Proposition \ref{main1} there exists an isometry $\Lambda$ making diagram \ref{4.2} commute and compatible with $f.$ Now apply Lemma \ref{main2}
to find $\Lambda'$ an isometry compatible with $f$ and making diagram \ref{4.2} commute, that is realized by a homeomorphism $F:X_{F_1}\to X_{F_2}.$

In either case, we obtain a homeomorphism over the surface exteriors $F:X_{F_1}\to X_{F_2}$ such that $(F|_{\partial},F_{\ast})\cong (f,\Lambda).$ We now invoke Proposition \ref{main3} to ensure that $\hat{F}:=F\cup H\text{ or } F\cup H'$ from $(X,F_1)\to (X,F_2)$ are isotopic to the identity. Therefore, $F_1$ and $F_2$ are topologically ambiently isotopic rel. boundary. 

\end{proof}
The proof above shows that, after extending a homeomorphism \(F_1\to F_2\) over the exterior and normal bundle, the resulting ambient homeomorphism of \(X\) is isotopic to the identity. We now turn to the analogous self-extension question: for a fixed surface \(F\subset X\), which homeomorphisms \(h:F\to F\) rel. boundary extend to homeomorphisms of the pair \((X,F)\)?

\begin{theorem}
    Let $K\subset S^3$ be a knot and $F\subset X$ a genus $g$ surface with boundary $K$ and $\pi_1(X\setminus F)=1.$ Fix an orientation-preserving homeomorphism $h: F \to F$ that restricts to the identity on $\partial F.$
    \begin{itemize}
        \item When $F$ is characteristic, if $h$ preserves the Rokhlin form $q_{FK}\circ h_{\ast}=q_{FK}$ then $h$ extends as a homeomorphism of pairs $H:(X,F)\xrightarrow[]{\cong}(X,F).$\item When $F$ is ordinary, $h$ extends as a homeomorphism of pairs $H:(X,F)\xrightarrow[]{\cong}(X,F).$ 
        \end{itemize}
\end{theorem}
\begin{proof} 

We first consider the case where $X_F$ is spin. If $F$ is characteristic, suppose $h$ preserves the Rokhlin quadratic form. Extend $h$ to a bundle isomorphism $H:(\overline{\nu}F,F)\to (\overline{\nu}F,F).$ When $X$ is spin proceed similarly: Theorem \ref{SpinU} provides a bundle isomorphism $H',$ agreeing with $h$ on the zero section, for which the induced boundary map identifies the extending spin structures. Now follow the proof of Theorem \ref{mainthm}, in the spin exterior case, to conclude. 
If $X_F$ is not spin, extend any homeomorphism rel. boundary $h$ to a bundle isomorphism $H:(\overline{\nu}F,F)\to (\overline{\nu}F,F).$ Now follow the proof of Theorem \ref{mainthm} to conclude.
\end{proof}

\subsection{Obtaining a morphism/compatible pair.} In this subsection, we define compatible pairs $(f,\Lambda)$ and give a procedure to construct such pairs by building on techniques analogous to Boyer's treatment of the closed case in \cite{Boyer1993}. The following results adapt the framework of Boyer to the setting of surfaces with boundary.

Let $M_0,M_1$ be simply-connected 4-manifolds with connected boundary. Any isometry $\Lambda:H_2(M_0)\to H_2(M_1)$ induces maps on $H_2(M_1,\partial M_1)\to H_2(M_0,\partial M_0)$ by taking the dual of $\Lambda$ and identifying
$\Hom_{\Z}(H_2(M_i),\Z)$ and $H_2(M_i,\partial M_i)$ for $i=0,1$ via the Poincaré duality and evaluation isomorphisms. Likewise, $\Lambda$ induces maps $H_i(\partial M_0)\to H_i(\partial M_1)$ by taking the unique maps making the following diagram commute: \begin{equation*}
    \begin{tikzcd}[ampersand replacement=\&,cramped]
	0 \& {H_2(\partial M_0)} \&\& {H_2(M_0)} \&\& {H_2(M_0,\partial M_0)} \&\& {H_1(\partial M_0)} \& 0 \\
	\\
	0 \& {H_2(\partial M_1)} \&\& {H_2(M_1)} \&\& {H_2(M_1,\partial M_1)} \&\& {H_1(\partial M_1)} \& 0
	\arrow[from=1-1, to=1-2]
	\arrow[from=1-2, to=1-4]
	\arrow[from=1-2, to=3-2]
	\arrow[from=1-4, to=1-6]
	\arrow["\Lambda"', from=1-4, to=3-4]
	\arrow[from=1-6, to=1-8]
	\arrow[from=1-8, to=1-9]
	\arrow[from=1-8, to=3-8]
	\arrow[from=3-1, to=3-2]
	\arrow[from=3-2, to=3-4]
	\arrow[from=3-4, to=3-6]
	\arrow["{{\Lambda^{\ast}}}"', from=3-6, to=1-6]
	\arrow[from=3-6, to=3-8]
	\arrow[from=3-8, to=3-9].
\end{tikzcd}
\end{equation*}
\begin{definition}\label{pair}
     Fix an orientation-preserving homeomorphism $f:\partial M_0\to \partial M_1,$ and an isometry $\Lambda:(H_2(M_0),Q_{M_0})\xrightarrow{\cong}(H_2(M_1),Q_{M_1}).$ We say that $(f,\Lambda)$ is a \emph{compatible pair} if $\Lambda$ and $f$ induced isomorphisms  $H_{\ast}(\partial M_0)\xrightarrow{\cong} H_{\ast}(\partial M_1)$ agree. The pair $(f,\Lambda)$ is \emph{realized by a homeomorphism} if there exists a homeomorphism $F:M_0\to M_1$ such that $(F|_{\partial},F_{\ast})= (f,\Lambda).$
\end{definition}

Once a compatible pair is obtained, the following results will give us sufficient conditions for when we can construct a homeomorphism of the surface exteriors, but some preparation is needed. For our purposes define $I^1(\partial X_{F_i})$ to be the image of $H^1(\partial X_{F_i};\Z)$ in $H^1(\partial X_{F_i};\Z_2)$ under mod 2 reduction. $\theta(f,\Lambda)$ is defined to be a class in $I^1(\partial X_{F_i})$ such that $\theta(f,\Lambda)\cdot \pi_{\Lambda}(\mathfrak{s}_1)=f_{\#}(\mathfrak{s}_1)$ for $\pi_{\Lambda}:\Spin(\partial X_{F_1})\to \Spin(\partial X_{F_2})$ the induced bijection by $\Lambda$ and $f_{\#}:\Spin(\partial X_{F_1})\to \Spin(\partial X_{F_2})$ the induced bijection by $f^{-1}$ \cite[Section 4]{boyer1986simply}.

\begin{customlemma}{Theorem}\cite[Theorem 0.7]{boyer1986simply}
Let $M_0$ and $M_1$ be simply-connected 4-manifolds with connected boundary and equal Kirby-Siebenmann invariants. A compatible pair $(f,\Lambda):\partial M_0\to \partial M_1$ is realized by a homeomorphism $F:M_0\to M_1 $ if and only if the obstruction $\theta(f,\Lambda)$ is trivial.
\end{customlemma}
\begin{customlemma}{Proposition}\cite[Proposition 0.8]{boyer1986simply}
 
    Let $M_0$ and $M_1$ be simply-connected 4-manifolds with connected boundary and equal Kirby-Siebenmann invariants. Given a compatible pair $(f,\Lambda)$ 
    \begin{enumerate}
        \item If $M_0$ and $M_1$ are non-spin then there exists an isometry $$\;\Lambda':(H_2(M_0),Q_{M_0})\xrightarrow{\cong}(H_2(M_1),Q_{M_1})$$ such that $(f,\Lambda')$ remains compatible and is realized by a homeomorphism.
        \item If $M_0$ and $M_1$ are spin, then $(f,\Lambda)$ is realized by a homeomorphism if and only if the manifold $M_0\cup_f-M_1$ is spin.
    \end{enumerate}
    \label{0.8}
\end{customlemma}

First recall we denote the unit sphere bundle of $F$ by $\mathring{Y},$ which, in our case, is homeomorphic to $F\times S^1.$ 
\begin{notation}
The following maps will be used in diagram \ref{4.2}.

    \begin{itemize}
        \item $(i_j)_{\ast}:H_2(X_{F_j})\to H_2(X)$ denotes the inclusion induced map 
        \item $\partial_{X_{F_j}}:H_3(X,X_{F_j})\to H_2(X_{F_j})$ denotes the connecting homomorphism with regard to the long exact sequence of the pair $(X,X_{F_j}).$
        \item $(exc)_j$ denotes the standard excision map with respect to the decomposition of $X$ as $X_{F_j}\cup \overline{\nu}F_j.$
        \item $\partial_{\mathring{Y}_j}:H_3(\overline{\nu}F_{j},\mathring{Y}_j)\to H_2(\mathring{Y}_j)$ denotes the connecting homomorphism with regard to the long exact sequence of the pair $(\overline{\nu}F_{j},\mathring{Y}_j).$
        \item $i_{\mathring{Y}_j}:H_2(\mathring{Y}_j)\to H_2(\partial X_{F_j})$ denotes the induced map by inclusion.
        \item $i_{\partial X_{F_j}}:H_2(\partial X_{F_j})\to H_2(X_{F_j})$ denotes the induced map by inclusion.
    \end{itemize}
\end{notation}
We now prove the first of the main intermediate results which we recall for the reader. 
\begin{figure}
   \label{4.2}
\end{figure}
\begin{customlemma}{Proposition 4.1}
For a given bundle isomorphism, $H:(\overline\nu F_1,F_1)\to  (\overline\nu F_2,F_2),$ fix $f:=H|_{\mathring{Y}}\cup id_{E_K}.$  There exists an isometry $\Lambda:H_2(X_{F_1})\to H_2(X_{F_2})$ such that the diagram below commutes and $(f,\Lambda)$ is a compatible pair.

\begin{equation}
\begin{tikzcd}[ampersand replacement=\&,cramped]
 	0 \& {H_3(X,X_{F_1})} \& {} \&\& {H_2(X_{F_1})} \&\& E([F_1]) \& 0 \\
	\\
	\& {H_3(\overline{\nu}F_1,\mathring{Y}_1)} \&\& {H_2(\partial X_{F_1})} \\
	\\
	\& {H_3(\overline{\nu}F_2,\mathring{Y}_2)} \&\& {H_2(\partial X_{F_2})} \\
	\\
	0 \& {H_3(X,X_{F_2})} \& {} \&\& {H_2(X_{F_2})} \&\& E([F_2]) \& 0
	\arrow[from=1-1, to=1-2]
	\arrow["{{\partial_{X_{F_1}}}}", from=1-2, to=1-5]
	\arrow["{{i_{1}}}", from=1-5, to=1-7]
	\arrow["\Lambda", from=1-5, to=7-5]
	\arrow[from=1-7, to=1-8]
	\arrow["{{{{id_E}}}}"', from=1-7, to=7-7]
    \arrow["="'', from=1-7, to=7-7]
	\arrow["\cong"', from=3-2, to=1-2]
    \arrow["exc_1"'', from=3-2, to=1-2]
	\arrow["{{{i_{\mathring{Y}_1}\circ\partial_{\mathring{Y}_1}}}}", from=3-2, to=3-4]
	\arrow["{{{H_{\ast}}}}"', from=3-2, to=5-2]
	\arrow["{{{{(i_{\partial X_{F_1}})_{\ast}}}}}", from=3-4, to=1-5]
	\arrow["{{{(H|_{\mathring{Y}}\cup id_{E_K})_{\ast}}}}"', from=3-4, to=5-4]
	\arrow["{{{i_{\mathring{Y}_2}\circ\partial_{\mathring{Y}_2}}}}", from=5-2, to=5-4]
	\arrow["\cong"'', from=5-2, to=7-2]
    \arrow["exc_2"', from=5-2, to=7-2]
	\arrow["{{{{(i_{\partial X_{F_2}})_{\ast}}}}}"', from=5-4, to=7-5]
	\arrow[from=7-1, to=7-2]
	\arrow["{{\partial_{X_{F_2}}}}", from=7-2, to=7-5]
	\arrow["{{i_{2}}}", from=7-5, to=7-7]
	\arrow[from=7-7, to=7-8].
\end{tikzcd}\tag{4.2}\end{equation}
\end{customlemma}

\begin{proof} Following the proof of \cite[Lemma 4.4]{Boyer1993},
we first construct an isomorphism between the second homology groups $\Lambda:H_2(X_{F_1})\to H_2(X_{F_2}).$ For convenience we recall the definition of $E([F_i]),$ given $\alpha\in H_2(X,\partial X)$ write $E(\alpha):=\{x\in H_2(X):Q_X(x, \alpha)=0\}.$ Since $F_1,F_2$ are homologous $E([F_1])=E([F_2])$ and we write this as $E.$

Since $E$ is free, the row splits. Pick splittings $s_1:H_2(X_{F_1})\to H_3(X,X_{F_1})$ and $s_2:E([F_2])\to H_2(X_{F_2})$ then define $\Lambda$ as follows: $$\Lambda(x)=\partial_2\circ (exc_2)\circ H_{\ast}\circ (exc_1)^{-1}\circ s_1(x)+s_2\circ 1_{E}\circ i_{1}(x).$$ 
We argue that this defines an isometry. Set:
\[
A := \partial_2 \circ \mathrm{exc}_2 \circ H_* \circ \mathrm{exc}_1^{-1} \circ s_1,
\qquad
B := s_2 \circ 1_E \circ i_1.
\]
Then \(\Lambda = A + B\), so
\[
Q_{X_{F_2}}(\Lambda(x),\Lambda(y))
= Q_{X_{F_2}}(Ax+Bx,\; Ay+By).
\]
By bilinearity,

$$Q_{X_{F_2}}(\Lambda(x),\Lambda(y))
= Q_{X_{F_2}}(Ax,Ay)
 + Q_{X_{F_2}}(Ax,By) 
 + Q_{X_{F_2}}(Bx,Ay)
 + Q_{X_{F_2}}(Bx,By).$$

By Proposition \ref{QX}, the first three terms vanish since $Q_{X_{F_2}}$ vanishes away from the $E[F_2]$ summand,
\[
Q_{X_{F_2}}(\Lambda(x),\Lambda(y))
= Q_{X_{F_2}}(Bx,By).
\]

Substituting \(B\) gives

\begin{align*}
    Q_{X_{F_2}}(Bx,By)
&= Q_{X_{F_2}}\!\left(
s_2 \circ i_1(x),\,
s_2 \circ  i_1(y)
\right)\\&=Q_{X}(i_2\circ s_2\circ i_1(x),i_2\circ s_2\circ i_1(y))
    \\& =Q_X|_E(i_1(x),i_1(y))
    \\&= Q_{X_{F_1}}(x,y).
\end{align*}

Now that $\Lambda$ is an isometry, we can see by construction $\Lambda$ induces $f_{\ast}:H_2(\partial X_{F_1})\to H_2(\partial X_{F_2}).$ Recall $\partial X_{F_i}$ is a closed 3-manifold and $f$ is a homeomorphism, hence $f_{\ast}$ must preserve the intersection pairing. Let $\partial(\Lambda)_i:H_i(\partial X_{F_1})\to H_i(\partial X_{F_2})$ denote homomorphisms induced by $\Lambda$ for $i=1,2.$ By the nonsingularity of the pairing \[\frac{H_1(\partial X_{F_1})}{TH_1(\partial X_{F_1})}\times H_2(\partial X_{F_1})\to \Z\] we now have that $\partial(\Lambda)_1$ is equal to $f_{\ast}$ up to torsion. Hence when the relative Euler number is zero we have our desired compatible pair that makes diagram \ref{4.2} commute.

Now assume $e\neq 0.$ Now by using the meridian of $F_i$ we construct closed surfaces $\Sigma_i$ in $X$ such that $[\Sigma_i]\cdot [F_i]=1\;(\operatorname{for}\;i=1,2).$ Pick a point on $F_i$ and consider the $D^2-$ fibre of that point in the tubular neighborhood. $X_{F_i}$ is simply-connected therefore the boundary of the fiber $\partial D$ is null-homotopic in the complement. Pick a surface $D_i'$ with $\partial D_i'=\partial D_i.$ Since their boundaries are identified we have constructed surfaces $\Sigma_i$ with $[\Sigma_i]\cdot [F_i]=1.$ 

We then see the image of $\Sigma_i$ under the boundary homomorphism $\partial_{\partial X_{F_i}}:H_2(X_{F_i},\partial X_{F_i})\to H_1(\partial X_{F_i})$ precomposed with $k_i:H_2(X)\to H_2(X_{F_i},\partial X_{F_i})$ is the $S^1-$fibre represented by $\{pt\}\times S^1\subset F_i\times S^1=\mathring{Y}_i.$ 

Moreover using successively that $f|_{\mathring{Y_1}}$ is a bundle isomorphism, hence preserving $S^1-$ fibers, Lemma \ref{kj}, and the definition of $\partial(\Lambda)_1,$ we obtain \begin{align*}
    f_{\ast}(\partial_{\partial X_{F_1}}\circ k_1([\Sigma_1]))&=\partial_{\partial X_{F_2}}(k_2([\Sigma_2]))\\&=\partial_{\partial X_{F_2}}\circ(\Lambda^{\ast})^{-1}(k_1([\Sigma_1]))\\&=\partial(\Lambda)_1(\partial_{\partial X_{F_1}}\circ k_1([\Sigma_1])). 
\end{align*}
Since $F_1$ and $F_2$ have the same nonzero relative Euler number $e,$ the torsion subgroup of $H_1(\partial X_{F_i}),$ denoted $TH_1(\partial X_{F_i}),$ is isomorphic to $\Z_e$ generated, geometrically, by the $S^1-$fiber, and above we have shown $\partial(\Lambda)_1$ and $f_{\ast}$ agree as maps on the torsion subgroup. 

Since we have shown that $\partial(\Lambda)_1$ and $f_{\ast
}$ agree up to torsion, we obtain $\partial (\Lambda)_1=f_{\ast}(x)+\psi(x)$ for some homomorphism $\psi:(H_1(\partial X_{F_1}),TH_1(\partial X_{F_1}))\to (TH_1(\partial X_{F_1}),0).$ Now we can defer to the proof of Proposition 1.6 from \cite[338-339]{boyer1986simply} for a systematic way to modify $\Lambda$ to obtain a $\Lambda'$ that satisfies $\partial(\Lambda')_i=f_{\ast}\text{ for }i=1,2,$ completing the proof.  
\end{proof}

The following Lemma guarantees we can always trivialize $\theta(f,\Lambda).$ Let $v\in H_2(X)$ such that $Q_X(v,[F])=1,$ and consider the map, $H_2(X)\xrightarrow[]{k}H_2(X,\mathring {Y})\cong H_2(X_{F_i},\mathring{Y})\oplus H_2(\overline{\nu}F_i,\mathring{Y}).$ 

 \begin{lemma}\label{4.7}
     For any compatible pair $(f,\Lambda)$ making diagram \ref{4.2} commute, the image of $k_2(v)\in H_2(X_{F_2},\partial X_{F_2})$ in $H_1(\partial X_{F_2}),$ given by the connecting homomorphism, is annihilated by $\theta(f,\Lambda)\pmod2.$ 
   
\end{lemma}
\begin{proof}
    We extend the proof of \cite[Lemma 4.7]{Boyer1993} to the case with boundary. Let $e,$ the relative Euler number of $F,$ be odd. Recall that in the proof of Proposition \ref{main1} we constructed a surface representing $v\in H_2(X).$ Under the composition $$H_2(X)\xrightarrow[]{k_2} H_2(X_{F_2},\partial X_{F_2})\xrightarrow[]{\partial_2}H_1(\partial X_{F_2})$$ the class $k_2(v)$ is represented by $\Sigma_2 \cap X_{F_2}$, a relative surface $\Sigma_2$ whose boundary is the $S^1$-fiber of $\mathring Y_2.$ Hence $\partial_2 k_2(v) \in H_1(\partial X_{F_2})$ generates the $\Z_e$-summand. When $e$ is odd, this summand has odd order and so the class vanishes after mod-2 reduction.

Now suppose then that $e$ is even. Fix a spin structure  $\mathfrak{s}_1\in \Spin(\overline{\nu}F_1)$ and restrict to $\mathfrak{s}^{\mathring{Y}}_1,$ a spin structure over $\mathring{Y}_1.$ Note that $\s_1$ extends to $\partial X_{F_1}$ because $e$ is even and therefore, the inclusion induces an isomorphism $H_1(\mathring{Y};\Z_2)\xrightarrow[]{\cong}H_1(\partial X_{F_1}).$ Since $f$ induces to a homeomorphism on $\mathring{Y}_1,$ define $\mathfrak{s}^{\mathring{Y}}_2:=f_{\#}(\mathfrak{s}^{\mathring{Y}}_1).$  
Let $$k^{\ast}:H^2(X_{F_i},\mathring{Y};\Z_2)\oplus H^2(\overline{\nu}F_i,\mathring{Y};\Z_2)\to H^2(X;\Z_2)$$ come from the relative Mayer-Vietoris map.

We first note the equality given in \cite[Page 348]{boyer1986simply} translates to $$w_2(X)=k^{\ast}(w_2(X_{F_i},\mathring{Y};\mathfrak{s}^{\mathring{Y}}_i),w_2(\overline{\nu}F_i,\mathring{Y};\mathfrak{s}^{\mathring{Y}}_i)).$$ 

Now we follow the original proof given by \cite[Lemma 4.7]{Boyer1993},
    \begin{align*}
    \langle\theta(f,\Lambda),\partial_2(k_2(v))\rangle&\equiv\langle\delta_2\theta(f,\Lambda),k_2(v)\rangle \text{\; for $\delta_2:H^1(\partial X_{F_2};\Z_2)\to H^2(X_{F_2};\Z_2)$ } \\&\equiv 
         \langle w_2(X_{F_2},\partial X_{F_2};f_{\#}(\mathfrak{s}_1))-w_2(X_{F_2},\partial X_{F_2};\pi_{\Lambda}(\mathfrak{s}_1)),k_2(v)\rangle
        \text{ By \cite[Proposition 4.1]{boyer1986simply} }\\& \equiv \langle w_2(X_{F_2},\partial X_{F_2};\mathfrak{s}_2),k_2(v)\rangle -\langle w_2(X_{F_1},\partial X_{F_1};\mathfrak{s}_1),\Lambda^{
        \ast} \circ k_2(v)\rangle 
        \\& \equiv \langle w_2(X_{F_2},\partial X_{F_2};\mathfrak{s}_2),k_2(v)\rangle-\langle w_2(X_{F_1},\partial X_{F_1};\mathfrak{s}_1),k_1(v)\rangle \text{ by Lemma \ref{4.3}}
     \\& \equiv \langle w_2(X_{F_2},\mathring{Y_2};\s^{\mathring{Y}}_2),k_2(v) \rangle -\langle w_2(X_{F_1},\mathring{Y}_1;\s^{\mathring{Y}}_1),k_2(v) \rangle
        \\& \equiv \langle \omega_2(X),v\rangle - \langle \omega_2(X),v\rangle\text{ by definition of $k^{\ast}$}
        \\& \equiv  0 \;(mod\; 2).    \end{align*}
Where the second to last equality follows from $w_2(\overline{\nu}F_2,\mathring{Y};\mathfrak{s}_2)\in H^2(\overline{\nu}F,\mathring{Y};\Z_2)$ being trivial since $\mathfrak{s}_2$ extends as spin structure over $\overline{\nu} F_2.$ 
\end{proof}
In a similar fashion to \cite[Page 45]{Boyer1993}, we show that when the exteriors are non-spin we can change $\Lambda$ in a systematic way.

Recall we identify the dual space of the second homology and the 2nd relative homology group via the Poincaré duality and evaluation isomorphisms: $$H_2(X_F,\partial X_F)\cong H^2(X_F)\cong \Hom_{\Z}(H_2(X_F);\Z).$$
\begin{lemma}\label{main2} Assume $X_{F_i}$ is non-spin. Fix a compatible pair $(f,\Lambda)$ making diagram \ref{4.2} commute. There exists an isometry  $\Lambda':H_2(X_{F_1})\to H_2(X_{F_2})$ such that $(f,\Lambda')$ remains compatible, makes diagram \ref{4.2} commute, and is realized by a homeomorphism $F:X_{F_1}\to X_{F_2}.$
\end{lemma}
\begin{proof}
Define $$\Lambda':=\Lambda+i_{\partial X_{F_2}}\circ \psi \circ \hat{Q}_{X_{F_1}}:H_2(X_{F_1})\to H_2(X_{F_2}).$$ As in \cite[Page 351]{boyer1986simply}, we note that for $\psi:H_2(X_{F_1},\partial X_{F_1})\to H_2(\partial X_{F_2})$ an arbitrary homomorphism $(f,\Lambda')$ is a compatible pair. We will choose $\psi$ to have image generated by a single element, $\beta\in H_2(\partial X_{F_2})$ to enforce the equality $\theta(f,\Lambda')=\theta(f,\Lambda)+PD^{-1}(\beta)\;(mod\;2)$ given in \cite[Page 351]{boyer1986simply}. We first check $\Lambda'$ makes diagram 4.2 commute. It suffices to show the following segment of the diagram \ref{4.2} commutes: 
\[\begin{tikzcd}[ampersand replacement=\&,cramped]
	{H_2(\partial X_{F_1})} \&\& {H_2(X_{F_1})} \&\& {E([F_1])} \\
	\\
	{H_2(\partial X_{F_2})} \&\& {H_2(X_{F_2})} \&\& {E([F_2])}
	\arrow["{(i_{\partial X_{F_1}})_{\ast}}", from=1-1, to=1-3]
	\arrow["{f_{\ast}}", from=1-1, to=3-1]
	\arrow["{i_1}", from=1-3, to=1-5]
	\arrow["{\Lambda '}", from=1-3, to=3-3]
	\arrow["{id_E}", from=1-5, to=3-5]
	\arrow["{(i_{\partial X_{F_2}})_{\ast}}", from=3-1, to=3-3]
	\arrow["{i_2}", from=3-3, to=3-5]
\end{tikzcd}\]
The left square commutes by the following: 
$$\Lambda'(i_{\partial X_{F_1}}(y))=\Lambda(i_{\partial X_{F_1}}(y))+i_{\partial X_{F_2}}\circ \psi \circ \hat{Q}_{X_{F_1}}(i_{\partial X_{F_1}}(y))=\Lambda(i_{\partial X_{F_1}}(y))=i_{\partial X_{F_1}}(f_{\ast}(y)).$$ The second to last equality holds because, since $\hat{Q}_{X_{F_1}}$ is the adjoint of $H_2(X_{F_1}),$ any element in $Im(i_{\partial X_{F_1}})$ evaluates to zero. Notice this did not depend on our choice of $\psi.$

The right square requires more care. When $e\neq 0$ we choose $\beta\in im(\psi)$ whose Poincaré dual reduces to $\theta(f,\Lambda)$ mod 2 to ensure $\theta(f,\Lambda')=0.$ By Lemma \ref{surj}, when $e\neq 0,$  $i_2\;\circ\; i_{\partial X_{F_2}}\equiv0,$ therefore, $i_2\circ\Lambda'=i_2(\Lambda+i_{\partial X_{F_2}}\circ \psi \circ \hat{Q}_{X_{F_1}})=i_2\circ\Lambda=i_1$ as desired. 

When $e=0,$ $i_2\;\circ\; i_{\partial X_{F_2}}$ need not be the zero map. We again pick a $\beta$ whose dual reduces to $\theta(f,\Lambda)$ mod 2, but additionally we require $Q_{\partial X_{F_2}}(\beta,[\mu])=0$ for $[\mu]=\partial_2(k_2(v))\in H_1(\partial X_{F_2}).$ 
Since the final $\Z$-summand of $H_2(\mathring{Y}) \cong H_2(\partial X_F) \cong \Z^{2g} \oplus \Z$ is generated by the Poincaré dual of $\mu$, the equality $Q_{\partial X_{F_2}}(\beta,[\mu])=0$ together with the diagram from $(1)$ ensure that $i_{\partial X_{F_2}}(\beta) \in \im(i_{\partial X_{F_2}})$ and thus $i_2(i_{\partial X_{F_2}}(\beta))=0$. Therefore, again  we have $i_2\circ\Lambda'=i_2\circ\Lambda=i_1$ completing the proof. 
\end{proof}
\begin{customlemma}{Proposition 4.2}
    If $F:X_{F_1}\to X_{F_2}$ restricted to the boundary agrees with $f:=H|_{\mathring{Y}}\cup id_{E_K}$ and $(f,\Lambda):=(F|_{\partial},F_{\ast})$ makes diagram \ref{4.2} commute then $\hat{F}=F\cup H:X\to X$ is isotopic to the identity.  
\end{customlemma}
\begin{proof}
   Consider the short exact sequence given by evaluation on $[F]$ from Lemma \ref{Im=E}:
   \[\begin{tikzcd}[ampersand replacement=\&,cramped]
	0 \& {E([F_1])} \&\& {H_2(X)} \&\& {\Z} \& 0 \\
	\\
	\\
	0 \& {(E[F_2])} \&\& {H_2(X)} \&\& {\Z} \& 0
	\arrow[from=1-1, to=1-2]
	\arrow[from=1-2, to=1-4]
	\arrow["{=}", from=1-2, to=4-2]
	\arrow["{Q_X([F_1],-)}", from=1-4, to=1-6]
	\arrow["{\hat{F}_{\ast}}", from=1-4, to=4-4]
	\arrow[from=1-6, to=1-7]
	\arrow["=", from=1-6, to=4-6]
	\arrow[from=4-1, to=4-2]
	\arrow[from=4-2, to=4-4]
	\arrow["{Q_X([F_2],-)}", from=4-4, to=4-6]
	\arrow[from=4-6, to=4-7]
\end{tikzcd}\]
For $\hat{F}$ to be isotopic to the identity it is sufficient to prove that $\hat{F}_{\ast}|_{H_2(X)}$ is the identity by \cite[Corollary C]{orson2022mapping}.
Consider the inclusion map $(i_j)_{\ast}:H_2(X_{F_j})\to H_2(X).$   
For $x \in H_2(X_{F_1})$, $\hat{F}_{\ast}(i_1(x))=i_{2}(F_{\ast}(x))$ since $\hat F$ extends $F$; this equals $i_2(\Lambda(x)) = i_1(x)$ by the commutativity of diagram \ref{4.2}.

Fix $\alpha=[F_1]=[F_2].$ For $x\in H_2(X)$ observe that $x-\hat{F}_{\ast}(x)$ is in $E(\alpha)$ by the following computation: 
\begin{align*}
    Q_X(x-\hat{F}_{\ast}(x),\alpha) &= Q_X(x,\alpha)-Q_X(\hat{F}_{\ast}(x),\alpha)
    \\&= Q_X(x,\alpha)-Q_X(\hat{F}_{\ast}(x),\hat{F}_{\ast}(\alpha)) \text{     since $\hat{F}_{\ast}(\alpha)=\alpha$}
    \\&=Q_X(x,\alpha)-Q_X(x,\alpha) \text{     since $\hat{F}_{\ast}$ is an isometry}
    \\&=0 .
\end{align*}
Hence $Q_X,\;x-\hat{F}_{\ast}(x)\in E(\alpha).$
Assume first that $F_i$ have nonzero relative Euler number. Then $Q_X|_{E(\alpha)}$ is non-degenerate by Lemma \ref{QXE}. Let $e\in E(\alpha),$ by the diagram above we know 
\begin{align*}  
    Q_X(x-\hat{F}_{\ast}(x),e) &= Q_X(x,e)-Q_X(\hat{F}_{\ast}(x),e)
    \\&= Q_X(x,e)-Q_X(\hat{F}_{\ast}(x),\hat{F}_{\ast}(e)) \text{ since LHS commutes}
    \\&=Q_X(x,e)-Q_X(x,e)
    \\&=0. 
\end{align*}

By non-degeneracy, $x-\hat{F}(x)=0 $ which implies $ \hat{F}(x)=x.$ Since $x\in H_2(X)$ was arbitrary when the relative Euler number of $F$ is nonzero, $\hat{F}$ induces the identity on $H_2(X)$ and $\hat{F}|_{\partial X}= id_{\partial X}$ therefore, by \cite[Corollary C]{orson2022mapping} $\hat{F}$ is isotopic to the identity rel. boundary and we conclude $F_1$ is topologically isotopic rel. boundary to $F_2.$

Now suppose the relative Euler number of $F_1,F_2$ is zero, recall there exists an element $v\in H_2(X)$ such that $Q_X(v,\alpha)=1.$ Assume that $\hat{F}_{\ast}(v)=v+m\tilde{\alpha}$ where $m\in \Z,\text{and }\tilde{\alpha}:=(\tilde{i}_{\ast})^{-1}(\alpha)$ for $\tilde{i}:H_2(X)\to H_2(X,\partial X).$ Then $v\cdot v=\hat{F}_{\ast}(v)\cdot \hat{F}_{\ast}(v)=v\cdot v +2m(v\cdot\tilde\alpha)+\tilde\alpha\cdot \tilde{\alpha}$ where $v\cdot \tilde\alpha=v\cdot \alpha=1$ and $\tilde\alpha\cdot \tilde \alpha=\tilde\alpha\cdot \alpha=e(F_i)=0$ by assumption. This forces $m=0$ hence $\hat{F}_{\ast}(v)=v.$

We now verify that this assumption indeed holds. By Lemma \ref{QXW} implies $\hat{F}_{\ast}(v)=v+m\tilde{\alpha}.$ Since $H_2(X)\cong \langle v\rangle \oplus E(\alpha),$ we can assume $\hat{F}_{\ast}(v)=qv+y$ for $q\in \Z$ and $y\in E(\alpha).$ Then the following equalities hold:
\begin{align*}
v\cdot \tilde{\alpha}
&= \hat{F}_{\ast}(v)\cdot \hat{F}_{\ast}(\tilde{\alpha})
&& \text{since $\hat{F}_{\ast}$ is an isometry} \\
&= \hat{F}_{\ast}(v)\cdot \tilde{\alpha}
&& \text{since $\hat{F}_{\ast}(\tilde{\alpha})
   = \hat{F}_{\ast}((i_{\ast})^{-1}[F_1])
   = (i_{\ast})^{-1}(\hat{F}_{\ast}([F_1]))
   = (i_{\ast})^{-1}[F_2]=\tilde{\alpha}$} \\
&= (qv+y)\cdot \tilde{\alpha} \\
&= q\,v\cdot \tilde{\alpha} + y\cdot \tilde{\alpha} \\
&= q\,v\cdot \tilde{\alpha} &&\text{ since by assumption $Q_X(y,\tilde{\alpha})=0.$}
\end{align*}
 This implies that $q=1,$ so $\hat{F}_{\ast}(v)=v+y.$ For $w\in E(\alpha)$ consider the following equalities: $$ v\cdot w= \hat{F}_{\ast}(v)\cdot\hat{F}_{\ast}(w)=(v+y)\cdot w=v\cdot w+y\cdot w.$$

Since $w\in E(\alpha)$ there exists $x\in H_2(X_{F_1})$ such that $w=i_1(x),$ and by the equality $\hat{F}_{\ast}(i_1(x))=i_{2}(\Lambda (x))=i_1(x)$ the second equality holds. 

We see that $y\cdot w=0$ for all $w\in W.$ Now to conclude the proof notice for $W$ given in Lemma \ref{QXW} by the splitting $E(\alpha)\cong  \langle \tilde{\alpha}\rangle \oplus W,$ $y\in E(\alpha)$ implies $y=y_w+m\tilde{\alpha}$ for $m\in\Z,$ and $y_w\in W.$ Since $y\cdot w=0$ for all $w\in W$ we see $0=y\cdot w= (y_w+m\tilde{\alpha})\cdot w=y_w\cdot w+m\tilde{\alpha}\cdot w$ where $m\tilde{\alpha}\cdot w=0 $ because $W\subset E(\alpha).$ Therefore $y_w\cdot w=0,$ so by the non-degeneracy of $Q_X|_W,\; y_w=0\in W.$ Therefore, $\hat{F}_{\ast}(v)=v+m\tilde{\alpha}$ as required. 

As $v$ and $E(\alpha)$ span $H_2(X),$ we have proved that $\hat{F}$ induces the identity on $H_2(X)$ and $\hat{F}|_{\partial X}= id_{\partial X}.$ Therefore, by \cite[Corollary C]{orson2022mapping} $\hat{F}$ is isotopic to the identity rel. boundary and we conclude $F_1$ is isotopic rel. boundary to $F_2.$ 
\end{proof}

\bibliographystyle{amsalpha}
\bibliography{main}
\end{document}